\documentclass[11pt]{article}
\usepackage[top=4cm,left=2.8cm,right=2.8cm,bottom=3.5cm]{geometry}
\usepackage{amsmath,amsthm}
\usepackage{amsfonts}
\usepackage{amssymb}
\usepackage{graphicx}
\usepackage{subcaption}
\usepackage{float}
\usepackage{relsize,exscale}
\usepackage{rotating}
\usepackage{psfrag}
\usepackage{color}
\usepackage{soul}
\usepackage{cite}
\usepackage{tikz}
\usetikzlibrary{er,positioning}
\usepackage{graphics}
\usepackage{graphicx}
\usepackage{algpseudocode,algorithm,algorithmicx}
\usepackage{xspace}
\usepackage{subcaption}
\usepackage{setspace}
\usepackage{verbatim}
\usepackage{fullwidth}
\usepackage{float}
\usepackage{listings}
\usepackage{lipsum}
\usepackage{framed}
\usepackage{caption}
\usepackage{booktabs,multirow,array}
\usepackage{multirow}
\usepackage{multicol}
\usepackage{mathtools}
\mathtoolsset{showonlyrefs,showmanualtags}
\usepackage{hyperref}
 \usepackage{array}
 \usepackage{hvfloat}
\usepackage{colortbl}
\usepackage[title]{appendix}
\usepackage{etoolbox}
\usepackage{bm}

\begin{document}
\newtheorem{thm}{Theorem}[section]
\newtheorem{lemma}{Lemma}[section]
\newtheorem{proposition}{Proposition}[section]
\newtheorem{rem}{Remark}[section]
\newtheorem{definition}{Definition}[section]
\title{\Large Parametric elliptic reconstructions and \emph{a posteriori} error estimates for parabolic partial differential equations with small randomness in a Robin boundary condition}

\author{Nakidi Shravani$^{1}$, Gujji Murali Mohan Reddy$^{1}$,
Amiya Kumar Pani$^{2}$, Stig Larsson$^{3}$\footnote{Corresponding author: stig@chalmers.se}
\\{\small $^{1}$Department of Mathematics,}\\{\small BITS-Pilani Hyderabad Campus,}\\{\small Telangana - 500078, India}
\\{\small $^{2}$Department of Mathematics,}\\{\small BITS-Pilani, KK Birla Goa Campus, NH 17 B, Zuarinagar,}\\{\small Goa - 403726, India}
\\{\small $^{3}$Department of Mathematical Sciences,}
\\{\small Chalmers University of Technology 
and University of Gothenburg,}\\{\small SE--412 96 Gothenburg, Sweden}
}
\date{}
\maketitle
\begin{abstract}
We obtain reliable \emph{a posteriori} residual-based error estimates for parabo\-lic partial differential equations with small randomness in a Robin-type boundary condition. The uncertainty is addressed via a perturbation approach, transforming the problem with small random input data into a sequence of deterministic problems. Finite element approximations, combined with the backward Euler time discretization, are employed for the resulting problems. To ensure optimal spatial accuracy, the elliptic reconstruction framework is suitably adapted to this setting. This is achieved by introducing a parametric elliptic reconstruction operator that unifies the \emph{a posteriori} analysis of deterministic parabolic problems with that of parabolic problems with small uncertainties.  The obtained \emph{a posteriori} error estimator is robust with respect to the parameter that describes the amount of uncertainty, in the sense that the constants appearing in the bounds are independent of the parameter, as well as of the mesh size and the time-step. In addition, numerical experiments are presented to validate the theoretical findings and illustrate the robustness of the proposed estimators.
\end{abstract}

\section{Introduction}\label{sec:1}
Mathematical models of diverse physical phenomena require adequate input data, including source terms, initial and boundary conditions, and coefficients, to provide a comprehensive study. In practice, the input data are prone to uncertainty due to challenges in characterizing the physical system, obtaining precise measurements, and handling complex geometries. Incorporating these uncertainties into the mathematical model yields partial differential equations (PDEs) with random input data. Specifically, in this article, we focus on parabolic PDEs with small random input data in a Robin-type boundary condition. 

Parabolic PDEs containing terms with small uncertainty, or subject to initial or boundary conditions with small uncertainty, can arise, for instance, in Newton's law of cooling  \cite{W99,W01,B11,GA98,D12} or glaciology \cite{J10, PRR04, PRR08}. While numerical methods offer a powerful tool for approximating solutions to these equations, the presence of small random perturbations in the input data adds an additional layer of complexity. Analyzing and quantifying errors arising from these uncertainties, along with discretization errors, requires rigorous techniques such as \emph{a posteriori} error estimation to ensure computational reliability.

The main aim of this study is to develop a rigorous framework for \emph{a posteriori} error analysis for parabolic PDEs with small randomness in a Robin-type boundary condition. For \emph{a posteriori} error estimation of elliptic and parabolic PDEs with random input data, we refer to \cite{BPRR19, CPB19, EGSZ14, EGSZ15, GNP16, G19, NV21, KPSZ23, MP22} and the references therein. In particular, very few articles have dealt with \emph{a posteriori} error estimation of a parabolic PDE with small random input data \cite{G19, SRM25}. 
Guignard \cite{G19} obtained residual-based \emph{a posteriori} error estimates for the backward Euler (BE) time discretization scheme in assessing the accuracy in the $H^1$-norm for physical space, the $L^2$-norm for time, and the $L^2$-norm for stochastic space for a linear parabolic PDE subject to a Robin boundary condition that contains a small uncertainty. Subsequently, the analysis was extended to the second-order Crank-Nicolson scheme in \cite{SRM25}. 

In this article, we adopt a probabilistic framework to model the uncertainty in the input data through random variables or random fields. In addition, the analysis involves a truncated Karhunen--Lo\`eve (KL) expansion \cite{L77,L78,W08,SGR89}, where we approximate a random field $\alpha$, representing the thermal coefficient in the Robin boundary of the problem, using a finite number of random variables. 
To deal with small uncertainty, we exploit the perturbation technique \cite{ZL04} to approximate the stochastic space. Perturbation techniques are easy to implement for small uncertainties, avoiding the curse of dimensionality and the high computational costs associated with other methods. This approach transforms the problem with small random input data into equivalent deterministic problems. Each problem is solved independently using the standard finite element method with piecewise linear elements for spatial discretization and advancing in time with the standard BE scheme.  

\textbf{Main contributions$\colon$} For the parabolic partial differential equations with small randomness in a Robin-type boundary condition, obtaining robust and sharp estimates in the weaker norms like $L^2$-norm in space remains largely an open problem. Addressing this gap is crucial for applications in computational physics, uncertainty quantification, and the development of robust adaptive discretization strategies. In particular, estimates are obtained by using ideas from \cite{MN03, LM06} in deterministic settings along with techniques from \cite{G19, SRM25} in stochastic frameworks. 

The main advancements in the present article are the following:
\begin{itemize}
    \item Novel residual-based and reliable \emph{a posteriori} error estimates are derived for the parabolic PDE with small randomness in a Robin boundary condition. These estimators allow spatial meshes to evolve over time and provide estimates in the $L^2$-norm for physical space, the $L^\infty$-norm for time, and the $L^2$-norm for stochastic space. In addition, the estimators naturally split into spatial, time, and stochastic discretization components. The analysis includes the derivation of first- and second-order estimates with respect to the parameter $\epsilon$, which describes the amount of uncertainty. 
     
    \item In order to obtain optimal order bounds in the spatial direction, parametric elliptic reconstruction operators are introduced and analyzed. We note that the elliptic reconstruction operator \cite{MN03, LM06} in conjunction with the energy technique has been used extensively for several deterministic time-dependent problems \cite{BKM12, KL13, GSR18, GSRJ16}. 
    
    \item The robustness of the resulting \emph{a posteriori} error estimators with respect to $\epsilon$, as well as their effectiveness, is demonstrated through numerical experiments. 
\end{itemize}

The remainder of the article is laid out as follows. In Section \ref{sec:2}, the model problem is introduced along with notation and preliminaries. Section \ref{sec:3} addresses the first order estimates in $\epsilon$ for the 
BE scheme, and $O(\epsilon^2)$ estimates are presented in Section \ref{sec:4}. In addition, the efficiency of the estimator is examined empirically in Section \ref{sec:5}. We conclude the article in Section \ref{sec:6}.
\section{Notation and preliminaries}\label{sec:2}
Let $D$ denote a bounded polygonal or polyhedral domain in $\mathbb{R}^d$, where $d=1, 2, 3$, with a Lipschitz boundary $\partial D$. For a Lebesgue measurable set $\mathcal{X} \subseteq D\subset \mathbb{R}^d$, we write $L^p(\mathcal{X})$, $1 \leq p < \infty$, to denote the space of equivalence classes of measurable functions $\psi$ on $\mathcal{X}$ such that $\int_\mathcal{X} |\psi(\mathbf{x})|^p \ d\mathbf{x} < \infty$. The norm $||\psi||_{L^p(\mathcal{X})}$ in this space is defined as follows:
 \begin{equation*}
 ||\psi||_{L^p(\mathcal{X})} = \left( \int_{\mathcal{X}}|\psi(\mathbf{x})| ^p \ d\mathbf{x} \right)^\frac{1}{p}, ~1 \leq p < \infty. 
 \end{equation*} 
 When $p=\infty$, the space $L^\infty(\mathcal{X})$ is the set of equivalence classes of essentially bounded measurable functions on $\mathcal{X}$, and the associated norm is given by
 \begin{equation*}
 ||\psi||_{L^\infty(\mathcal{X})}= \text{ess~sup}_{\mathbf{x}\in \mathcal{X}} |\psi(\mathbf{x})|.
 \end{equation*}
 For $m \in \mathbb{N}$ and $1\leq p < \infty$, the Sobolev space $W^{m,p}(\mathcal{X})$ of order $m$ 
 on $\mathcal{X}$ is defined as 
 \begin{equation*}
 W^{m,p}(\mathcal{X})=\{\psi \in L^p(\mathcal{X}) : \mathcal{D}^{\Upsilon} \psi \in L^p(\mathcal{X}), ~0 \leq |\Upsilon| \leq m \},
  \end{equation*} 
 and is endowed with the norm
 \begin{eqnarray}
  ||\psi||_{ W^{m,p}(\mathcal{X})} 
 =\left(\underset{0 \leq |\Upsilon| \leq m}{\sum} ||\mathcal{D}^{\Upsilon} \psi||_{L^p(\mathcal{X})}^p  \right)^{\frac{1}{p}} \nonumber, ~ 1 \leq p < \infty.
  \end{eqnarray}
   For $p=2$, $W^{m,p}(\mathcal{X})=H^m(\mathcal{X})$ is a Hilbert space with the inner-product
  \begin{equation*}
   \langle \phi, \psi \rangle_{H^m(\mathcal{X})}= \underset{0 \leq |\Upsilon| \leq m}{\sum} \int_{\mathcal{X}} \mathcal{D}^{\Upsilon}\phi \mathcal{D}^{\Upsilon} \psi \ d\mathbf{x},~~ \phi, \psi \in H^m(\mathcal{X}),
 \end{equation*} 
  and the induced norm \(||\psi||_{H^m(\mathcal{X})} \) is given by
  \begin{eqnarray}
  ||\psi||_{H^m(\mathcal{X})} 
 =\left(\underset{0 \leq |\Upsilon| \leq m}{\sum} ||\mathcal{D}^{\Upsilon} \psi||_{L^2(\mathcal{X})}^2  \right)^{\frac{1}{2}} \nonumber.
  \end{eqnarray}
In particular, when $m = 1, 2$, and \(\mathcal{X} = D\), we denote the 
norm on \(H^m(\mathcal{X})\) by \(||\psi||_{m}\). The corresponding 
semi-norm is denoted by \(|\psi|_{m}\). Also, when $m=0$, and $\mathcal{X}=D$ we omit the subscript on the inner product and its corresponding norm. 
We now define the Bochner space \( L^\infty(0,T; L^2(D)) \), which consists of functions $v(t, \mathbf{x})$ that vary in both space and time and is given by
\[
L^\infty(0,T; L^2(D)) = \left\{ v \colon v(t, \cdot) \in L^2(D) \text{ a.e. } t \in [0,T], \text{ess\,sup}_{t \in [0,T]} \| v(t, \cdot) \| < \infty \right\}.\]
Moreover, the norm on this space is defined by
\begin{align*}
\| v \|_{L^\infty(0,T; L^2(D))} 
 &:= \text{ess\,sup}_{t \in [0,T]} \| v(t, \cdot) \|. 
\end{align*}
Similarly, for a Lebesgue measurable set $\mathcal{X} \subseteq D$, the spaces of square-integrable functions over \( [0, T] \), with values in \( L^2(\mathcal{X}) \) and \( L^2(\gamma) \), respectively, are represented by \( L^2(0,T; L^2(\mathcal{X})) \) and $L^{2}(0,T;L^{2}(\gamma))$, where $\gamma$ denotes the boundary of  $\mathcal{X}$. 
The space $C^{0}([0,T];L^{2}(D))$ consists of all $L^{2}(D)$-valued continuous functions $u$ on $[0,T]$.

\par Now, we introduce probabilistic space-time function spaces. Let $(\Omega,F,P)$ be a complete probability space, where $\Omega$ is the sample
space, $F \subset2^{\Omega}$ is the $\sigma$-algebra of events, and $P$ is a
probability measure. For $1 \leq p < \infty$, $L_P^{p}(\Omega)$ represents the space of $p$-integrable random variables $Y$ on $(\Omega,F,P)$ and is defined as follows:
 \begin{equation*}
 L_P^{p}(\Omega)=\left\{Y(\omega): \int_{\Omega}|Y(\omega)|^p \, dP(\omega)\ < \infty\right\}.
 \end{equation*}
 If $Y \in L^{1}_P(\Omega)$, we denote its expected value by $E[Y]=\int_{\Omega} Y(\omega)\,dP(\omega)$. The space\\
 $L_P^{2}(\Omega;L^{\infty}(0,T;L^2(D)))$ is defined as
\begin{align*}
L_P^{2}(\Omega;L^{\infty}(0,T;L^2(D))) := \left \{ v \colon E\Big[  \text{ess~sup} \| v(t,\cdot,\cdot) \|^2 \Big] < \infty \right \}.
\end{align*} 
Henceforth, we will omit the notation $d\mathbf{x}$ and $ds$ whenever a function is integrated over a physical domain or its boundary, respectively.
\subsection{Problem settings}\label{subsect:2:1}
Consider the following parabolic PDE with random boundary conditions, where the input data is characterized by small uncertainty. Seek a random function $u \colon (0,T) \times D \times \Omega \rightarrow \mathbb{R}$ such that almost surely (a.s.) in $\Omega$ 
\begin{equation}
 \dfrac{\partial u(t,\textbf{x},\omega)}{\partial t}-\nabla \cdot ( k(\textbf{x})\nabla u(t,\textbf{x},\omega)) = f(t,\textbf{x}), \;\; \textbf{x} \in D,~ t \in (0,T), \label{spde}
\end{equation}
subject to Dirichlet, random Robin boundary conditions, and initial conditions
\begin{align}
  u(t,\textbf{x},\omega)&= 0, ~~\textbf{x} \in \Gamma_{\mathrm{D}}, ~t \in (0,T), \nonumber\\
   (k(\textbf{x}) \nabla u(t,\textbf{x},\omega)) \cdot \mathbf{n}
   + \alpha(\textbf{x}, \omega) u(t,\textbf{x},\omega)&= g(t,\textbf{x}),~\textbf{x} \in \Gamma_{\mathrm{R}}, \  t \in (0, T), \label{ibc}\\
 u(0,\textbf{x},\omega) &= u_{\text{init}}(\textbf{x}), ~~~~~\textbf{x} \in D, \nonumber
 \end{align}
 where $\nabla$ is the spatial gradient operator. 
 Here, $\Gamma_{\mathrm{D}}$ and $\Gamma_{\mathrm{R}}$ represent the Dirichlet and Robin boundary components, respectively, with $\Gamma_{\mathrm{D}} \cup \Gamma_{\mathrm{R}} = \partial D$ and $\Gamma_{\mathrm{D}} \cap \Gamma_{\mathrm{R}} = \emptyset$. In addition, the unit outward normal vector in $\Gamma_{\mathrm{R}}$ is denoted by $\mathbf{n}$.

Moreover, the input data for the PDE \eqref{spde}--\eqref{ibc} are subject to the following assumptions:
\begin{itemize}
\item 
The source function 
$f \in L^2(0,T;L^2(D))$, 
the boundary function $g\in L^2(0,T;L^2(\Gamma_{\mathrm{R}}))$ and the initial data $u_{\text{init}} \in L^2(D).$ 
\item The spatially dependent coefficient matrix $k= [{k}_{i,j}(\mathbf{x})]_{1\leq i,j\leq d}$ 
with $k_{i,j} \in L^{\infty}(D)$ satisfies the following uniform ellipticity condition: 
\begin{equation}
\exists ~ {k}_{\mathrm{min}}>0~ \text{such that}~\forall \varsigma \in \mathbb{R}^d,~ {k}(\mathbf{x})\varsigma \cdot \varsigma \geq {k}_{\mathrm{min}}|\varsigma|^2,~ \text{a.e.}~ \text{in}~ D.\label{kmin1} 
\end{equation}
\item The random field $\alpha(\cdot, \omega)$ is a.s. in $L^\infty(\Gamma_{\mathrm{R}})$ and is expressed as an affine combination of a finite number of random variables \cite{DKLS16}:
    \begin{equation}
        \alpha(\mathbf{x}, \mathbf{Y}(\omega)) = \alpha_0(\mathbf{x}) + \epsilon \sum_{j=1}^{L} \alpha_j(\mathbf{x}) Y_j(\omega), \label{alpha}
    \end{equation}
    where $\mathbf{Y} = (Y_1, Y_2, \ldots, Y_L)$ is a random vector with linearly independent random variables having zero mean and unit variance. 
\end{itemize}
We note that given the mean and covariance or two-point correlation of $\alpha$, a KL type expansion is used to obtain \eqref{alpha} \cite{L77,L78,GWZ14, GS03}. 

In light of the Doob–Dynkin Lemma \cite{BC02} and \eqref{alpha}, we rewrite the solution $u$ of \eqref{spde}--\eqref{ibc} as (with a slight abuse of notation)
\begin{align}
u(t,\textbf{x},\omega)=u(t,\textbf{x}, \mathbf{Y}(\omega))=u(t,\textbf{x}, Y_1(\omega),\cdots, Y_L(\omega)).
\end{align}
With  $\Gamma_j=Y_j(\Omega)\subset\mathbb{R}$, the range of the random variable $Y_j$, we set $\Gamma=\prod_{j=1}^{L}\Gamma_j \subset{\mathbb{R}^L}$. Furthermore, exploiting the independence of the random variables $Y_j$, the joint probability density function (PDF) $\rho\colon\Gamma \rightarrow \mathbb{R}^+$ of $\mathbf{Y}$ is defined as $\rho(\mathbf{y})=\prod_{j=1}^{L}\rho_j(y_j)$, $\mathbf{y}\in \Gamma$. Here, $\rho_j$, $j=1, \cdots, L$, denotes the PDF defined on $\Gamma_j$. Subsequently, $(\Omega,\mathcal{F}, P)$ can be substituted with $(\Gamma, B(\Gamma), \rho(\mathbf{y})\,d\mathbf{y})$, where $B(\Gamma)$ refers to the Borel $\sigma$-algebra on $\Gamma$ and $\rho(\mathbf{y})\,d\mathbf{y}$ denotes the probability measure of $\mathbf{Y}$.
Now define 
\begin{equation*}
V=H_{\Gamma_{\mathrm{D}}}^1=\left\{\phi \in H^1(D) : ~\phi=0~ \text{on}~ \Gamma_{\mathrm{D}}\right\}
\end{equation*}
with norm 
	\begin{align}
	||\phi||_{V}:=
	\begin{cases}
	|\phi|_{H^1(D)}=||\nabla \phi||, ~~~~~~~~~~~~~~~~~~~~~~~~~~~ &\text{if}~ \Gamma_{\mathrm{D}} \neq \emptyset, \\
	||\phi||_{H^1(D)} = \sqrt{||\phi||^2 ~+~ ||\nabla \phi||^2,} ~~~ &\text{if}~  \Gamma_{\mathrm{D}}= \emptyset.
	\end{cases}
	\end{align}
\par
 \noindent
 Consequently, the Bochner space \( L^2(0,T; V) \) represents the space of square integrable functions on \( [0, T] \) with values in \( V \). Now introduce the space $L_{\rho}^{2}(\Gamma;L^{2}(0,T;V)$ as
\[
L_{\rho}^{2}(\Gamma;L^{2}(0,T;V)) = \left\{v \colon  \Gamma\rightarrow L^{2}(0,T;V),\ v
\text{~is~strongly~measurable~and~} \|v\|_{L_{\rho}^{2}(\Gamma;L^{2}(0,T;V))}
< \infty\right\},
\]
where
\[
\|v\|^{2}_{ L_{\rho}^{2}(\Gamma;L^{2}(0,T;V))} = \int_{\Gamma} \|v(\cdot, \cdot;\mathbf{y})\|^{2}%
_{L^{2}(0,T;V)} \,\rho(\mathbf{y})\,d\mathbf{y}.
\]
The parametric point-wise variational formulation of problem \eqref{spde}--\eqref{ibc} reads as:
find \\ $u \in L_{\rho}^2(\Gamma;L^2(0,T;V) \cap C^0([0,T];L^2(D)))$  such that
 \begin{align}
 \begin{aligned}
\left\langle\dfrac{\partial u}{\partial t}, \phi \right\rangle 
+ a(u,\phi;\textbf{y}) &= F(\phi) ,\quad&& \text{a.e.} ~ t \in (0,T),~\rho\text{-a.e.}~\textbf{y} \in \Gamma, \;\;\forall \phi \in V,  \label{wfu} \\
~~~~~~~~~~~~~~~u(0,\textbf{x},\textbf{y}) &=u_{\text{init}}(\textbf{x}), && \textbf{x} \in D,~ \rho\text{-a.e.}~ \textbf{y} \in \Gamma,  
\end{aligned}
\end{align}
where
\begin{align}
a(u,\phi;\textbf{y}) &:=\int_{D} (k \nabla u )\cdot \nabla \phi + \int_{\Gamma _{R}} \alpha(\cdot,\textbf{y})~u\phi \label{pbl}, \\
F(\phi) &:= \int_{D} f\phi + \int_{ \Gamma_{\mathrm{R}}}g\phi \;\;\mbox{and}\;\;
\left\langle\dfrac{\partial u}{\partial t}, \phi \right\rangle := \int_{D}\dfrac{\partial u}{\partial t} \phi. \label{plf-1}
\end{align}
Moreover, under the following assumption,  
\begin{equation}
\alpha(\textbf{x},\mathbf{y}) \geq \alpha_{\text{min}}>0~~\text{a.e}.~\textbf{x} \in \Gamma_{\mathrm{R}},~\rho\text{-a.e.} ~\mathbf{y} \in \Gamma, \label{alphamin}
\end{equation}
\eqref{wfu} is well-posed.
In addition, the coercivity of the bilinear form $a(\cdot, \cdot; \mathbf{y})$ is ensured by \eqref{kmin1} and \eqref{alphamin}, i.e., there exists $C_a >0$ such that
\begin{equation}
C_{a}||v||_{V}^2 \leq a(v,v;\mathbf{y}) \;\; \forall v \in V ~\text{and}~ \rho\text{-}\text{a.e.}~ \mathbf{y} \in \Gamma. \label{coer}
\end{equation}
In addition, there exists a constant $\beta>0$ such that
\begin{align}
    a(v,\phi;\mathbf{y}) \leq \beta \|v\|_V\|\phi\|_V\;\;\forall v,\phi \in V ~\text{and}~ \rho\text{-}\text{a.e.}~ \mathbf{y} \in \Gamma. \label{bdd}
\end{align}
\par
Now, we deploy the perturbation technique to deal with small uncertainty in the input data. Under suitable regularity and smallness assumptions \cite{GS03}, the solution $u$ admits the perturbation expansion 
\begin{equation}
u(t,\textbf{x},\mathbf{Y}(\omega))=u_0(t,\textbf{x})+ \epsilon u_1(t,\textbf{x},\mathbf{Y}(\omega))+  \epsilon^2 u_2(t,\textbf{x},\mathbf{Y}(\omega))+\cdots, \label{uexp}
\end{equation}
where $\epsilon$ is the parameter that denotes the size of the uncertainty in \eqref{alpha}. 

\section{First-order approximation in \texorpdfstring{$\epsilon$}{}}\label{sec:3}
We first reduce the problem \eqref{spde}--\eqref{ibc} to deterministic problems using the perturbation method \cite{G19}. In view of \eqref{spde}--\eqref{ibc} and \eqref{alpha}, the problem associated with the first term $u_0$ in the expansion of \eqref{uexp} can be written as: seek 
$u_0 \colon (0, T) \times D \rightarrow \mathbb{R}$ such that
 \begin{align}
  \begin{aligned}
 \dfrac{\partial u_0(t,\textbf{x})}{\partial t}- \nabla \cdot (k(\textbf{x}) \nabla u_0(t,\textbf{x}))&= f(t,\textbf{x}), \quad &&\textbf{x} \in D,~~t \in (0,T), \\
 u_0(t,\textbf{x})&=0, && \textbf{x} \in \Gamma_{\mathrm{D}},~t \in (0,T),\\
 (k(\textbf{x}) \nabla u_0(t,\textbf{x})) \cdot \mathbf{n}
 +\alpha_0(\textbf{x}) u_0(t,\textbf{x})&= g(t,\textbf{x}),&&\textbf{x} \in \Gamma_{\mathrm{R}},~~t \in (0,T), \label{det}  \\
 u_0(0,\textbf{x})&=u_{\text{init}}(\textbf{x}), &&\textbf{x} \in D.
  \end{aligned}
 \end{align}
 The corresponding variational formulation reads as:
 find $u_0 \in L^2(0,T;V) \cap C^0([0,T];L^2(D))$ such that
 \begin{align}
 	\mathop{\mathlarger{\int_{D}}}\dfrac{\partial u_0}{\partial t}\phi + \int_{D}(k \nabla u_0) \cdot \nabla \phi + \int_{\Gamma_{\mathrm{R}}} \alpha_0 u_0 \phi &= \int _{D}f\phi + \int_{\Gamma _{R}} g\phi, \;\;\forall \phi \in V, ~t \in (0,T),\label{wf1}\\
 u_0(0,\textbf{x}) &=u_{\text{init}}(\textbf{x}), ~~\textbf{x}\in D.\nonumber
 \end{align}
We treat the above problem with finite elements in space, and the backward Euler (BE) scheme in the temporal direction. Let $0 = t_0 < t_1 < \cdots < t_N = T$ be a partition of $[0,T]$ with $I_n := [t_{n-1}, t_n]$,
the $n$-th subinterval of length $\tau_n:= t_n - t_{n-1}$. 

Given a family of shape-regular, conforming triangulations $\mathcal{T}_n,~ n=0, 1, 2, \cdots, N,$ 
of the domain $D$, let $h_n$ be the local mesh-size function defined as 
\begin{equation*}
h_n(\mathbf{x}) = \text{diam}(K),\;\; K \in \mathcal{T}_n\;\; \text{and}\;\; \mathbf{x} \in K.
\end{equation*}
Let $\Xi_n$ and  $\mathcal{E}_n$ denote the set of all sides (edges) and internal sides of the triangulation $\mathcal{T}_n$, respectively. Furthermore, let $\Sigma_n$ be the union of all sides, i.e., $\bigcup_{E \in \Xi_n }E$.
For each $\mathcal{T}_n$, we associate continuous, piecewise linear polynomial spaces  defined as 
\begin{equation*}
\tilde{V}^n=\left\{v \in C^0(\bar{D}) \colon v|_K \in \mathbb{P}_1 \;\forall K \in \mathcal{T}_n\right\},~~\text{and}~~V^n=\tilde{V}^n\cap V. 
\end{equation*}
where $\mathbb{P}_1$ is the space of polynomials of degree $\leq 1$. We note that the analysis in this article also holds for $C^0$-piecewise higher-degree polynomials. 

 We define $\hat{h}_n:=\max(h_n, h_{n-1})$, for two successive compatible \cite{LM06} triangulations $\mathcal{T}_{n-1}$ and $\mathcal{T}_{n}$. 
In addition, we use  sets $\hat{\Sigma}_n:=\Sigma_n \cap\Sigma_{n-1}$ and $\check{\Sigma}_n:={\Sigma}_n\cup{\Sigma}_{n-1}.$ 
For $v_n \in V^n, n \in [0,N]$, we set
\begin{align}
{{\bar{\partial}}} v_n := \frac{v_n - v_{n-1}}{\tau_n}\;\;\; \text{and}\;\;\;
   \widehat{\partial} v_n :=\frac{v_n-P_0^nv_{n-1}}{\tau_n}, \;\;\; n = 1,2,\cdots,N.\nonumber
\end{align}

We now state the BE scheme as follows: set $u^0_{0,h}:=I^0u_{\text{init}}$, find $u_{0,h}^n \in V^n$, $n = 1,2,\cdots, N,$  such that
\begin{equation}
	\int_{D} {\bar{\partial} u_{0,h}^n} \phi_n+ \int_{D} (k \nabla u_{0,h}^n) \cdot \nabla \phi_n + \int_{ \Gamma_{\mathrm{R}}} \alpha_0 u_{0,h}^n \phi_n  = \int_{D} f^{n}\phi_n  + \int_{ \Gamma_{\mathrm{R}}} g^{n}\phi_n 	~ \forall \phi_n \in V^n,\label{BE}	
\end{equation}
where $f^n=f(\cdot,t_n)$, $g^n=g(\cdot,t_n)$ and $I^0$ is a suitable interpolation or projection operator from $V$ or $L_2(D)$ onto $V^0$. 
\par
The $L_2$ projection operator $P_0^n \colon L_2(\Omega) \rightarrow {\tilde{V}}^n$  is defined as
\begin{equation}
    \langle P_0^nv, \phi_n\rangle =\langle v, \phi_n \rangle~~ \forall\phi_n \in {\tilde{V}}^n.
\end{equation}
Further, we define the $L_2$ projection operator on $\Gamma_{\mathrm{R}}$ in the sense of trace,  
$P_{\frac{1}{2}}^n \colon L_2(\Gamma_{\mathrm{R}}) \rightarrow {\tilde{{V}}}^n$ as 
\begin{equation}
    \int_{\Gamma_{\mathrm{R}}} P_{\frac{1}{2}}^nv \phi_n =\int_{\Gamma_{\mathrm{R}}}  v \phi_n \;\; \forall\phi_n \in {\tilde{V}}^n.
\end{equation}

The discrete elliptic operator $A^n\colon \tilde{V}^n \rightarrow \tilde{V}^n$ is defined by 
\begin{align}
    \langle A^n v , \phi_n\rangle= a(v, \phi_n; \cdot) \;\; \forall \phi_n \in \tilde{V}^n~ \text{and}~\rho\text{-a.e.~in~} \Gamma.
\end{align}
Note that $a(\cdot, \cdot; \cdot)$ denotes $a(\cdot, \cdot; \mathbf{y})$ with implicit dependence on $\mathbf{y}$.  
\par
To build residuals, we use a particular representation of the bilinear form $a(\cdot, \cdot; \cdot)$. For $v \in V^n$ and $\rho$-a.e. in $\Gamma$, we express
\begin{align}
    a(v, \phi; \cdot)= \langle A_{\textit{el}}v, \phi \rangle+\langle J[v], \phi \rangle_{{\Sigma}_{n}} \;\; \forall \phi \in V, \label{RepErec}
\end{align}
where 
\begin{align}
\langle A_{\text{el}}v, \phi \rangle= \sum_{K \in \mathcal{T}_n}\langle -\nabla \cdot (k \nabla v), \phi \rangle_K ~~ \forall \phi \in V
\end{align}
and
\begin{equation}
J[v]|_E :=
\begin{cases} \label{jump0}
\frac{1}{2}[(k \nabla v) \cdot \mathbf{n}_E]_{\mathbf{n}_E},~~~~~~~~ &\text{if}~~~ E \subset D,\\
  \alpha v+ (k \nabla v)\cdot \mathbf{n}_E, ~~~~~~ &\text{if}~~~E \subset \Gamma_{\mathrm{R}},  \\
 0, ~~~~~~~~~~~~~~~~~~~~~~~~~~~~~~~~~~ &\text{if}~~~ E \subset \Gamma_{\mathrm{D}}. 
\end{cases}
\end{equation}
Here $[v]_{\mathbf{n}_E}$, the spatial jump of the field $v$ across an interior side $E \in \mathcal{E}_n$, is given by
\begin{align}
[v]_{\bm{n}_E}(\mathbf{x}):=\lim_{t \rightarrow 0^+}(v(\mathbf{x}+ t {\mathbf{n}_E})(\mathbf{x})- v(\mathbf{x}- t {\mathbf{n}_E})(\mathbf{x})),
\end{align}
where ${\mathbf{n}_E}(\mathbf{x})$ is a normal vector to $E$ at the point $\mathbf{x}$.

Now, we introduce a parametric elliptic reconstruction operator.
 \begin{definition} The parametric elliptic reconstruction operator $\mathcal{R}^n\colon V^n \rightarrow V$ is defined as 
 \begin{align}
     a(\mathcal{R}^nv, \phi; \cdot)=\langle A^nv, \phi\rangle~~\forall \phi \in V~ \text{and}~\rho\text{-a.e.~in}~ \Gamma. \label{ERec}
 \end{align}
 \end{definition}
Further, we obtain the following orthogonality property by virtue of the above definition
 \begin{align}
     a(v-\mathcal{R}^nv,\phi_n; \cdot)=0~~ \forall \phi_n \in V^n ~ \text{and}~\rho\text{-a.e.~in}~ \Gamma.
 \end{align}
Associated with the nodal values $u_{0,h}^{n-1}$ and $u_{0,h}^n$, $n = 1, 2, \cdots, N$, we define a continuous linear interpolant in time $u_{0,h \tau}$ by
\begin{align}
u_{0, h \tau}(t) :=  l_{n-1}(t) u_{0,h}^{n-1}+l_{n}(t) u_{0,h}^n, ~~ t \in I_n, \label{linapx1} 
\end{align}
where
\begin{align}
l_{n-1}(t):=\frac{t_n - t}{\tau_n}, \;\;\;\;\;\;\;~~ ~~
l_{n}(t):= \frac{t-t_{n-1}}{\tau_n}. \;\;\;\;\;\;\;\label{timebasis}
\end{align}
Similarly, the linear interpolant $u_{\mathcal{R},0}(t)$ associated with the reconstructed nodal values $\mathcal{R}^{n-1}u^{n-1}_{0,h}$ and $\mathcal{R}^nu^{n}_{0,h}$, $n = 1, 2, \cdots, N$ is given by
\begin{align}
  u_{\mathcal{R},0}(t):=l_{n-1}(t)\mathcal{R}^{n-1}u_{0,h}^{n-1}+l_{n}(t)\mathcal{R}^{n}u_{0,h}^{n}, ~~ t \in I_n. \label{linapxw0} 
\end{align}

\subsection*{{\it{A posteriori}} error estimators}
We begin by introducing the residuals.  
For the problem \eqref{det}, the element residual $\mathfrak{R}_0^n$ and jump residual $\mathfrak{J}_0^n$ are defined  as
\begin{equation}
\mathfrak{R}_0^n|_K:=\mathfrak{f}^n-{\widehat{\partial} u_{0,h}^n}+ \nabla \cdot (k \nabla u_{0,h}^n)\label{res1}
\end{equation}
and
\begin{equation}
\mathfrak{J}_0^n|_E :=
\begin{cases} \label{jump1}
\frac{1}{2}[(k \nabla u_{0,h}^n)\cdot \mathbf{n}_E]_{\mathbf{n}_E},~~~~~~~~ &\text{if}~~~ E \subset D,\\
 \mathfrak{g}^n- \alpha_0 u_{0,h}^n- (k \nabla u_{0, h}^n)\cdot \mathbf{n}_E, ~~~~~~ &\text{if}~~~E \subset \Gamma_{\mathrm{R}},  \\
 0, ~~~~~~~~~~~~~~~~~~~~~~~~~~~~~~~~~~ &\text{if}~~~ E \subset \Gamma_{\mathrm{D}}, 
\end{cases}
\end{equation}
where
\begin{align}
\mathfrak{f}^n:=P_0^nf^n, \;\;\;\;\;\; \mathfrak{g}^n:=P_{\frac{1}{2}}^ng^n.
\end{align}
Now, we develop preliminary results to estimate the error $e_0=u_{0,h\tau}-u$. Set $e_0=\varrho_0-\varepsilon_0$, where $\varrho_0={u}_{\mathcal{R},0}-u$, and $\varepsilon_0={u}_{\mathcal{R},0}-u_{0,h \tau}$. Also, note that in the forthcoming analysis, all equations are valid for a.e. $t$ and a.s. in the sample space $\Omega$, without explicit mention.

In the following, we denote by superscripts R, S, T, St, D, DM,  the elliptic reconstruction, spatial, temporal, stochastic truncation, data approximation, data and mesh change errors, respectively. 

We begin with an estimate of the parametric elliptic reconstruction error. 
\begin{lemma}
For any $v \in V^n$, we have the following estimate of the elliptic reconstruction error.
\begin{align}
    \|\mathcal{R}^nu_{0,h}^n-u_{0,h}^n\| \leq 
 \zeta_{1,n}^\mathrm{R},
\end{align}
where 
\begin{align}
\zeta_{1,n}^\mathrm{R}:=C_4\left (\|h_n^2\mathfrak{R}_0^n\|+\|h_n^{\frac{3}{2}}J
_0^n\|\right), \label{Elp1}
\end{align}
where the constant depends on the interpolation constant and the domain.
\end{lemma}

\begin{proof}
The result mirrors \cite[Lemma 2.3]{LM06} using the standard Aubin--Nitsche duality argument \cite{AO97, V96}, and the details of the proof are therefore omitted for conciseness.  
\end{proof}

Now, we require the following several auxiliary lemmas to estimate the parabolic error $\varrho_0$.
\begin{lemma}\label{Lemma1} The following equation holds for $n=1,\dots, N$ and for all $\phi \in V \colon$ 
\begin{align}
    \int_D \widehat{\partial}u^n_{0,h} \phi + \int_DA^nu^n_{0,h} \phi-\int_{\Gamma_{\mathrm{R}}}(\alpha-\alpha_0)u^n_{0,h} \phi - \int_D \mathfrak{f}^n \phi- \int_{\Gamma_{\mathrm{R}}}\mathfrak{g}^n \phi =0. \label{DF1}
\end{align}
\end{lemma}
\begin{proof}
Using projection operators and the BE scheme \eqref{BE}, we obtain the desired result. This completes the proof.
\end{proof}

\begin{lemma}\label{Lemma2} The following error equation for $\rho_0$ holds for $n= 1, 2, \cdots, N$ and $~\forall \phi \in V$
\begin{align}
\int_{D}\dfrac{\partial \varrho_0}{\partial t} \phi + a(\varrho_{0}, \phi; \cdot) =& \int_{D} \dfrac{\partial \varepsilon_0}{\partial t} \phi +a({u}_{\mathcal{R},0}-u_{\mathcal{R},0}^n, \phi; \cdot)+\int_{D} (\mathfrak{f}^n-f) \phi
    +\int_{ \Gamma_{\mathrm{R}}}(\mathfrak{g}^n-g) \phi \nonumber\\
    &+ \frac{1}{\tau_n}\int_D (P_0^n-I)u_{0,h}^{n-1} \phi+\int_{\Gamma_{\mathrm{R}}} (\alpha-\alpha_0) u_{0,h}^n \phi,\label{MPErr}
\end{align}
where $u_{\mathcal{R},0}^n = u_{\mathcal{R},0}(t_n)$.
\end{lemma}
\begin{proof}
For $\phi \in V$, using \eqref{ERec}, \eqref{DF1} can be written as
\begin{align}
  \int_D {\bar{\partial}}u^n_{0,h} \phi +& a(u_{\mathcal{R},0}^n,\phi;\cdot)-\int_{\Gamma_{\mathrm{R}}}(\alpha-\alpha_0)u^n_{0,h} \phi - \int_D \mathfrak{f}^n \phi  \nonumber\\
    &- \int_{\Gamma_{\mathrm{R}}}\mathfrak{g}^n \phi -\frac{1}{\tau_n} \int_D 
    (P_0^n - I) u_{0,h}^{n-1}\phi=0. \label{PWE1}
\end{align}
In view of \eqref{wfu}, we find that
 \begin{align}
    \int_{D}\dfrac{\partial \varrho_0}{\partial t} \phi + a(\varrho_0, \phi; \cdot) = &\int_{D} \frac{\partial u_{\mathcal{R},0}}{\partial t} \phi +a(u_{\mathcal{R},0}, \phi; \cdot) - \int_{D} f \phi
    -\int_{ \Gamma_{\mathrm{R}}}g \phi.  \label{1MTh1}
\end{align}
By subtracting \eqref{PWE1} from \eqref{1MTh1}, we obtain the desired result. \end{proof}

We now estimate the parabolic error $\varrho_0(t)$ in the following lemma. 
\begin{lemma}\label{Lemma3}
For $m =1,2,\cdots, N$, the following estimate holds:
\begin{align}
\left(\max_{t \in [0,t_m]}\|\varrho_0(t)\|^2+2C_a \int_{0}^{t_m}\|\varrho_0(t)\|_V^2dt\right)^{\frac{1}{2}}\leq 2\|\varrho_0(0)\|+2\sqrt{2}({\sigma^2_{1,m}}+{\sigma^2_{2,m}})^{\frac{1}{2}}, \label{Lemma3.3}
\end{align}
where
\begin{align}
\sigma_{1,m}^2:=2\left(\sum_{n=1}^m\tau_n\left(\zeta_{1,n}^\mathrm{S}+\zeta_{1,n}^\mathrm{T}+\zeta_{1,n}^{\mathrm{D}_1}\right)\right)^2,\;\;\sigma_{2,m}^2:=\frac{1}{C_a}\sum_{n=1}^m{\tau_n}\left(\zeta_{1,n}^{\mathrm{D}_2}+\zeta_{1,n}^\mathrm{DM}+\zeta_{1,n,\mathbf{y}}^\mathrm{St}\right)^2.\label{1Thsigma1} 
\end{align}
Moreover,
\begin{align}
\zeta_{1,n}^\mathrm{S} &:= C_1 \left(\|\hat{h}_n^2 \bar{\partial} \mathfrak{R}_0^n\|+\|\hat{h}_n^{\frac{3}{2}} \bar{\partial} \mathfrak{J}_0^n\|_{\hat{\Sigma}_n}+\|\hat{h}_n^{\frac{3}{2}} \bar{\partial} \mathfrak{J}_0^n\|_{{\check{\Sigma}_n}/{\hat{\Sigma}_n}}\right), \label{1ThS}\\
\zeta_{1,n}^\mathrm{T} &:= \|A^{n-1}u_{0,h}^{n-1}-A^{n}u_{0,h}^{n}\|, \label{1ThT} \\
\zeta_{1,n}^{\mathrm{D}_1} &:= \frac{1}{\tau_n}\int_{t_{n-1}}^{t_n}\|f^n-f\|dt, \label{dataf1}\\
\zeta_{1,n}^{\mathrm{D}_2} &:= C_{2}\left(\frac{1}{\tau_n}\int_{t_{n-1}}^{t_n}\|g^n-g\|_{L^2(\Gamma_{\mathrm{R}})}^2dt\right)^{\frac{1}{2}}, \label{datag1}\\
\zeta_{1,n}^\mathrm{DM} &:= C_3\left(\left\|h_n(P_0^n - I)\left(f^n+\frac{u_{0,h}^{n-1}}{\tau_n}\right)\right\|+\left\|(P_{\frac{1}{2}}^n-I)g^n\right\|_{L^2(\Gamma_{\mathrm{R}})}\right),\label{1ThDM}\\
\zeta_{1,n,\mathbf{y}}^\mathrm{St} &:= C_2\left(\frac{1}{\tau_n}\int_{t_{n-1}}^{t_n}\|(\alpha-\alpha_0)u_{0,h}^n\|^2_{L^2(\Gamma_{\mathrm{R}})}dt \right)^{\frac{1}{2}}. \label{1ThSt}
\end{align}
Furthermore, the constants appearing in the estimators depend on the trace constant, Poincar\'e constant, and the mesh aspect ratio.
\end{lemma}
\begin{proof} 
A choice of  $\phi=\varrho_0$ in \eqref{MPErr} with  \eqref{coer} shows that
\begin{align}
\frac{1}{2} \frac{d}{dt}\|\varrho_0\|^2+C_a \|\varrho_0\|^2_V 
\leq& \left|\int_{D} \dfrac{\partial \varepsilon_0}{\partial t} \varrho_0 \right| +|a(u_{\mathcal{R},0}-u_{\mathcal{R},0}^n, \varrho_0; \cdot)|+\left|\int_{D} ({f}^n-f) \varrho_0 \right|+\left|\int_{ \Gamma_{\mathrm{R}}}\left({g}^n-g\right) \varrho_0\right| \nonumber\\
&+ \left|\int_D (P_0^n-I)\left(\frac{u_{0,h}^{n-1}}{\tau_n}+f^n\right) \varrho_0\right|+\left|\int_{\Gamma_{\mathrm{R}}} (P_{\frac{1}{2}}^n-I)g^n\varrho_0\right| \nonumber\\
& +\left|\int_{\Gamma_{\mathrm{R}}} (\alpha-\alpha_0) u_{0,h}^n \varrho_0\right|. \label{1Th3}
\end{align}
Integrating from $0$ to $t_m$ leads to 
\begin{align}
\frac{1}{2}\|\varrho_0(t_m)\|^2+&{C_a}\int_0^{t_m}  \|\varrho_0\|^2_V dt\leq
 \frac{1}{2}\|\varrho_0 (0)\|^2+I_m, 
\end{align}
where $I_m:= \sum_{n=1}^m \left(I_n^1+I_n^2+I_n^3+I_n^4+I_n^5\right)$ and
\begin{align}
I_n^1:= &\int_{t_{n-1}}^{t_n}\left|\left\langle\dfrac{\partial \varepsilon_0}{\partial t}, \varrho_0 \right\rangle \right| dt,\nonumber\\
I_n^2:=&\int_{t_{n-1}}^{t_n}|a(u_{\mathcal{R},0}-u_{\mathcal{R},0}^n, \varrho_0; \cdot)| dt, \nonumber\\
I_n^3:=&\int_{t_{n-1}}^{t_n}\left(\left|\int_{D} (f^n-f) \varrho_0 \right|
    +\left|\int_{ \Gamma_{\mathrm{R}}}\left(g^n-g\right) \varrho_0\right|\right)dt, \nonumber\\
I_n^4:=&\int_{t_{n-1}}^{t_n} \left(\left|\int_D (P_0^n-I)\left(\frac{u_{0,h}^{n-1}}{\tau_n}+f^n\right) \varrho_0\right|+\left|\int_{\Gamma_{\mathrm{R}}} (P_{\frac{1}{2}}^n-I)g^n\varrho_0\right|\right)dt, \nonumber\\
I_n^5:=& \int_{t_{n-1}}^{t_n} \left|\int_{\Gamma_{\mathrm{R}}} (\alpha-\alpha_0) u_{0,h }^n \varrho_0\right|dt.  \nonumber
\end{align}Setting $\|\varrho_0( {t_m^*})\|^2=\max_{t \in [0, t_m]}\|\varrho_0(t)\|^2$, $t_m^* \in [0,t_m]$, we have 
\begin{align}
\frac{1}{2}\|\varrho_0(t_m^*)\|^2+ C_a \int_0^{t_m} \|\varrho_0\|_V^2 dt \leq \|\varrho_0(0)\|^2+2I_m.  \label{auxrho0}
\end{align}
The estimate for $\varrho_0$ relies on the following auxiliary results.\\
\textbf{Bound on $I_n^1\colon$ (Spatial error estimate)} 
Using the Aubin--Nitsche duality technique combined with the orthogonality property of the parametric elliptic reconstruction, we follow the approach of \cite[Section 3.6]{LM06} to obtain
\begin{align}
\sum_{n=1}^m I_n^1 \leq \|\varrho_0( {t_m^*})\|\sum_{n=1}^m \zeta_{1,n}^\mathrm{S}\tau_n, \label{1Th8}
\end{align}
where $\zeta_{1,n}^\mathrm{S}$ is given by \eqref{1ThS}. The proof is, therefore, omitted for brevity.\\
\textbf{Bound on $I_n^2 \colon$ (Time error estimate)} A use of  \eqref{linapxw0} along with \eqref{ERec} yields
\begin{align}
     I_n^2   \leq \int_{t_{n-1}}^{t_n}l_{n-1}(t)\|A^{n-1}u_{0,h}^{n-1}-A^{n}u_{0,h}^{n}\| \|\varrho_0(t)\|dt.\nonumber
\end{align}
Thus,
\begin{align}
\sum_{n=1}^m I_n^2 \leq \|\varrho_0( {t_m^*})\|\sum_{n=1}^m \zeta_{1,n}^\mathrm{T} \tau_n,
\end{align}
where $\zeta_{1,n}^\mathrm{T}$ is given in \eqref{1ThT}. \\
\textbf{Bound on $I_n^3 \colon$ (Data approximation estimate)} Using the trace inequality, there holds
\begin{align}
I_n^3
\leq  \int_{t_{n-1}}^{t_n} \|f^n-f\|\|\varrho_0\|dt +C_2\left(\int_{t_{n-1}}^{t_n}\|g^n-g\|_{L^2(\Gamma_{\mathrm{R}})}^2dt\right)^{\frac{1}{2}}\left(\int_{t_{n-1}}^{t_n}\|\varrho_0\|_V^2dt\right)^{\frac{1}{2}}, \nonumber
\end{align}
where $$C_2 = 
\begin{cases}
C_\mathrm{T}, & \text{if } \Gamma_{\mathrm{D}} = \emptyset, \\[4pt]
C_\mathrm{T} \sqrt{1 + C_\mathrm{F}^2}, & \text{otherwise},
\end{cases}$$
with $C_\mathrm{T}$ and $C_\mathrm{F}$, denoting the trace and Poincare constants, respectively. 
Therefore,
\begin{align}
    \sum_{n=1}^m I_n^3 \leq \|\varrho_0( {t_m^*})\|\sum_{n=1}^m \zeta_{1,n}^{\mathrm{D}_1}\tau_n +\sum_{n=1}^m \zeta_{1,n}^{\mathrm{D}_2}\tau_n^{\frac{1}{2}} \left(\int_{t_{n-1}}^{t_n}\|\varrho_0\|_V^2dt\right)^{\frac{1}{2}} ,
\end{align}
where $\zeta_{1,n}^{\mathrm{D}_1}$ and $\zeta_{1,n}^{\mathrm{D}_2}$ are given in \eqref{dataf1} and \eqref{datag1}, respectively.\\
\textbf{Bound on $I_n^4 \colon$ (Data approximation and mesh change estimator)}
Noting that $V^n \subset$ Ker $(P_0^n-I)$, and using the interpolation estimate, we arrive at
\begin{align}
I_n^4 \leq& C_{\Pi} \tau_n^{\frac{1}{2}}\left\|h_n(P_0^n-I)\left(\frac{u_{0,h}^{n-1}}{\tau_n}+f^n\right)\right\|\left(\int_{t_{n-1}}^{t_n}\|\varrho_0\|_V^2dt\right)^{\frac{1}{2}} \nonumber\\
  &+ C_2\tau_n^{\frac{1}{2}}\left\|(P_{\frac{1}{2}}^n-I)g^n\right\|_{L^2(\Gamma_{\mathrm{R}})}
 \left(\int_{t_{n-1}}^{t_n}\|\varrho_0\|_V^2 dt\right)^{\frac{1}{2}},
 \end{align}
where $C_{\Pi}$ denotes the interpolation constant.
With $C_3=\max\{C_{\Pi},C_2\}$, we obtain
\begin{align}
\sum_{n=1}^{m}I_n^4 \leq \sum_{n=1}^{m}\tau_n^{\frac{1}{2}}\zeta_{1,n}^\mathrm{DM}\left(\int_{t_{n-1}}^{t_n}\|\varrho_0\|_V^2 dt\right)^{\frac{1}{2}},
\end{align}
where $\zeta_{1,n}^\mathrm{DM}$ is given in \eqref{1ThDM}. \\
\textbf{Bound on $I_n^5 \colon$(Truncation error)}
Similar to the estimate of $\sum_{n=1}^m I_n^4$, the term $\sum_{n=1}^m I_n^5$ can be bounded as 
\begin{align}
\sum_{n=1}^m I_n^5 \leq \sum_{n=1}^m \tau_n^{\frac{1}{2}}\zeta_{1,n,\mathbf{y}}^\mathrm{St}\left(\int_{t_{n-1}}^{t_n}\|\varrho_0\|_V^2dt\right)^{\frac{1}{2}}, \nonumber
\end{align}
where $\zeta_{1,n,\mathbf{y}}^\mathrm{St}$ is given in \eqref{1ThSt}.

Finally, incorporating all the bounds in \eqref{auxrho0}, we arrive at
\begin{align}
\frac{1}{2}\|\varrho_0( {t_m^*})\|^2+C_a\int_0^{t_m}\|\varrho_0\|_V^2 dt \leq & \|\varrho_0(0)\|^2+2\|\varrho_0( {t_m^*})\| \sum_{n=1}^m\left(\zeta_{1,n}^\mathrm{S}+\zeta_{1,n}^\mathrm{T}+\zeta_{1,n}^{\mathrm{D}_1}\right)\tau_n\\
&+2\sum_{n=1}^m\left(\int_{t_{n-1}}^{t_{n}}\|\varrho_0\|_V^2 dt\right)^\frac{1}{2}\left(\zeta_{1,n}^{\mathrm{D}_2}+\zeta_{1,n}^\mathrm{DM}+\zeta_{1,n,\mathbf{y}}^\mathrm{St}\right)\tau_n^{\frac{1}{2}}. \nonumber
\end{align}
Using Young's inequality, a standard kick-back argument yields
\begin{align*}
\frac12 \|\varrho_0(t^*_m)\|^2 + C_a \int_0^{t_m} \|\varrho_0\|_V^2 \, dt
&\le \|\varrho_0(0)\|^2 + \frac14 \|\varrho_0(t^*_m)\|^2
    + 4 \left( \sum_{n=1}^{m} \left(\zeta^\mathrm{S}_{1,n} + \zeta^\mathrm{T}_{1,n}
               + \zeta^{D1}_{1,n}\right) \tau_n\right)^2 \\
&\quad + \frac{C_a}{2}\sum_{n=1}^{m} \int_{t_{n-1}}^{t_n} \|\varrho_0\|_V^2 \, dt 
        + \frac{2}{C_a} \sum_{n=1}^{m}
               \left(\zeta^{D2}_{1,n} + \zeta^\mathrm{DM}_{1,n} + \zeta^\mathrm{St}_{1,n,y}\right)^2\tau_n,
\end{align*}
which implies
\begin{align*}
\frac14 \|\varrho_0(t^*_m)\|^2 + \frac12 C_a \int_0^{t_m} \|\varrho_0\|_V^2 \, dt
&\le \|\varrho_0(0)\|^2 
    + 4 \left( \sum_{n=1}^{m} \left(\zeta^\mathrm{S}_{1,n} + \zeta^\mathrm{T}_{1,n}
               + \zeta^{D1}_{1,n}\right) \tau_n\right)^2 \\
&\quad         + \frac{2}{C_a} \sum_{n=1}^{m}
               \left(\zeta^{D2}_{1,n} + \zeta^\mathrm{DM}_{1,n} + \zeta^\mathrm{St}_{1,n,y}\right)^2\tau_n.
\end{align*}
Thus, we obtain the estimate for the parabolic error.
\end{proof}

We now state the main result in the following theorem.
\begin{thm}
Let $u$ be the solution of problem \eqref{wfu} and $u_{0,h \tau}$ be defined as in \eqref{linapx1}. 
Then, the following inequality holds 
\begin{align}
\left(E\left[\max_{t \in [0,t_m]}\|e_0(t)\|^2\right]\right)^{\frac{1}{2}} \leq \left(16E\Big[\|\varrho_0(0)\|^2\Big]+2\max_{n \in [0;m]}(\zeta_{1,n}^\mathrm{R})^2+32({\sigma^2_{1,m}}+ \sigma^2_{3,m})\right)^{\frac{1}{2}}, \label{Thm1}
\end{align}
where ${\sigma^2_{1,m}}$ is defined in Lemma \ref{Lemma3}, 
\begin{align}
\sigma^2_{3,m}:=&\frac{3}{C_a}\left(\sum_{n=1}^{m}\tau_n(\zeta_{1,n}^{\mathrm{D}_2})^2+\sum_{n=1}^{m}\tau_n( \zeta_{1,n}^\mathrm{DM})^2 +\sum_{n=1}^{m}\tau_n(\zeta_{1,n}^\mathrm{St})^2\right)
\end{align}
with
\begin{align}
(\zeta_{1,n}^{n,St})^2:=C_4\epsilon^2\sum_{j=1}^{L}\|\alpha_ju_{0,h\tau}\|^2_{L^2(\Gamma_{\mathrm{R}})}\label{1eta-St}.
\end{align}
In addition, $\zeta_{1,n}^{\mathrm{D}_2}$ and $\zeta_{1,n}^\mathrm{DM}$ are given in \eqref{datag1} and \eqref{1ThDM}, respectively. Furthermore, $C_4$ depends on the mesh aspect ratio and $C_a$ is the coercivity constant.
\end{thm}
\begin{proof} 
For a.e. $t \in (0,T)$, we note that 
\begin{align}
E\Big[\max_{t \in [0,t_m]}\|e_0(t)\|^2\Big] \leq 2E\Big[\max_{t \in [0,t_m]}\|\varepsilon_0(t)\|^2+ \max_{t \in [0,t_m]}\|\varrho_0(t)\|^2\Big]. \label{1Th3u0:ch:4} 
\end{align}
Using \eqref{linapx1} and \eqref{linapxw0}, the bound on the first term can be obtained as 
\begin{align}
E\Big[\max_{t \in [0,t_m]}\|\varepsilon_0(t)\|^2\Big] \leq E\Big[\max_{n \in [0;m]} (\zeta_{1,n}^\mathrm{R})^2\Big]=\max_{n \in [0;m]} (\zeta_{1,n}^\mathrm{R})^2, \label{1Th29}
\end{align}
where $\zeta_{1,n}^\mathrm{R}$ is given in \eqref{Elp1}. Lemma \ref{Lemma3} yields
\begin{align}
\max_{t \in [0,t_m]}\|\varrho_0(t)\|^2 &\leq 
8\|\varrho_0(0)\|^2+16({\sigma^2_{1,m}}+{\sigma^2_{2,m}}).
 \nonumber
\end{align}
Applying the expectations on both sides results in
\begin{align}
E\left[\max_{t \in [0,t_m]}\|\varrho_0(t)\|^2\right] & \leq 8E\left[\|\varrho_0(0)\|^2\right]+16\left (\sigma_{1,m}^2+E\Big[\sigma_{2,m}^2\Big]\right),\label{1Th32}
\end{align}
where
\begin{align}
E\left[\sigma_{2,m}^2\right]=\frac{3}{C_a}\left (\sum_{n=1}^{m}\tau_n(\zeta_{1,n}^{\mathrm{D}_2})^2+\sum_{n=1}^{m}\tau_n(\zeta_{1,n}^\mathrm{DM})^2+\sum_{n=1}^{m}\tau_nE(\zeta_{1,n,\mathbf{y}}^\mathrm{St})^2\right).\label{1Th34}
\end{align}
Now,  we obtain the desired result by substituting \eqref{1Th29} and \eqref{1Th32} in \eqref{1Th3u0:ch:4} and applying the following identity
\begin{equation}
E[Y_iY_j]=\delta_{ij}, \label{EY2}
\end{equation}
where $Y_i$'s are independent random variables with zero mean and unit variance, and $\delta_{ij}$ denotes the Kronecker delta.  
\end{proof}
 \section{Second-order approximation in \texorpdfstring{$\epsilon$}{}}\label{sec:4}
 This section aims to obtain the second-order estimates in $\epsilon$. To this end, it is required to solve the $L$ problems to obtain $u_1(t,\textbf{x}, \mathbf{Y}(\omega))=\sum_{j=1}^{L}U_j(t,\textbf{x})Y_j(\omega)$. We seek
 $U_j \colon (0,T) \times D \rightarrow \mathbb{R}$ such that 
  \begin{align}
  \begin{aligned}
  \frac{\partial U_j(t,\textbf{x})}{ \partial t}- \nabla \cdot (k(\textbf{x}) \nabla U_j(t,\textbf{x}))& =0, \quad && \mathbf{x} \in D, ~ t \in (0,T),  \\
  U_j(t,\textbf{x})&=0,&&\mathbf{x} \in \Gamma_{\mathrm{D}},~ t \in (0,T), \label{Uj} \\
  (k(\textbf{x}) \nabla U_j(t,\textbf{x})) \cdot \mathbf{n}
 + \alpha_0(\textbf{x})U_j(t,\textbf{x}) &= - \alpha_j(\textbf{x})u_0(t,\mathbf{x}),&& \mathbf{x} \in \Gamma_{\mathrm{R}},~t \in (0,T),  \\
 U_j(0,\textbf{x})&=0,&& \mathbf{x} \in D. 
  \end{aligned}
  \end{align}
 The corresponding variational formulation is to find $U_j \in L^2(0,T;V) \cap C^0([0,T];L^2(D)),~ j = 1, 2, \cdots, L,$ such that
  \begin{align}
     \int_{D} \dfrac{\partial U_j}{\partial t}\phi + \int_{D}(k \nabla U_{j}) \cdot \nabla \phi + \int_{\Gamma_{\mathrm{R}}} \alpha_0 U_j \phi  
     &= -\int_{\Gamma _{R}} \alpha_j u_0 \phi , \;\;\forall \phi \in V ,~\text{a.e.}~ t \in (0,T) \label{weakUj}\\
       U_j (0,\mathbf{x})&= 0,~~~~~~~~~~~~~~~~~\mathbf{x} \in D. \nonumber
 \end{align}
Given $U^0_{j,h}=0$, the BE nodal approximations $U_{j,h}^n \in V^n$, $n = 1,2,\cdots, N,$ satisfy
\begin{equation}
	\int_{D} \bar{\partial} U_{j,h}^n  \phi_n  + \int_{D} (k \nabla U_{j,h}^n) \cdot \nabla \phi_n + \int_{ \Gamma_{\mathrm{R}}} \alpha_0 U_{j,h}^n \phi_n  = -\int_{ \Gamma_{\mathrm{R}}} \alpha_j u_{0,h}^n \phi_n  ~~ \forall \phi_n \in V^n.\label{BE2}	
\end{equation}
Similarly to \eqref{linapx1}, $U_{j,h \tau}$ is defined to be a linear interpolant associated with the nodal values $U_{j,h}^{n-1}$ and $U_{j,h}^n$, $n=1,2, \cdots, N,$, i.e.,
\begin{eqnarray}
U_{j, h \tau}(t) &:=&  l_{n-1}(t) U_{j,h}^{n-1}+l_{n}(t) U_{j,h}^n,~ t \in I_n. \label{linapx2} 
\end{eqnarray}
Furthermore, let ${U}_{\mathcal{R},j}(t)$ be the linear interpolant associated with $\mathcal{R}^{n-1}U_{j,h}^{n-1}$ and $\mathcal{R}^{n}U_{j,h}^{n}$ i.e., 
\begin{align}
    {U}_{\mathcal{R},j}(t):=l_{n-1}(t)\mathcal{R}^{n-1}U_{j,h}^{n-1}+l_{n}(t)\mathcal{R}^{n}U_{j,h}^{n} ~ \text{for}~ t \in I_n~ \text{and}~ n=1,2, \cdots, N, \label{linapxwj} 
\end{align}
where $\mathcal{R}^{n}U_{j,h}^{n}$ is the elliptic reconstruction of $U_{j,h}^{n}$.
\par

\subsection*{{\it{A posteriori}} error analysis} 
For $j = 1, 2, \cdots, L$, we now introduce $\mathfrak{R}^n_{1,j}$ (element residuals) and  $\mathfrak{J}_{1,j}^n$ jump residuals for \eqref{Uj} as 
 \begin{equation}
	\mathfrak{R}_{1,j}^n|_K:= -{\widehat{\partial} U_{j,h}^n}+ \nabla \cdot (k \nabla U_{j,h}^n)\; \;\; \ \label{res2}
\end{equation}
and
\begin{equation}
\mathfrak{J}_{1,j}^n|_E :=
\begin{cases} \label{jump2}
\frac{1}{2}[(k \nabla U_{j,h}^n)\cdot  \textbf{n}_E]_{\textbf{n}_E}~~~~~~~~ &\text{if}~~~ E \subset D,\\
-\alpha_ju_{0,h}^n - \alpha_0 U_{j,h}^n- (k \nabla U_{j, h}^n) \cdot \textbf{n}_E ~~~~~~ &\text{if}~~~E \subset \Gamma_{\mathrm{R}}, \\
 0 ~~~~~~~~~~~~~~~~~~~~~~~~~~~~~~~~~~ &\text{if}~~~ E \subset \Gamma_{\mathrm{D}}.
\end{cases}
\end{equation}
Setting $u^1_{h\tau}:=u_{0,h\tau}+\epsilon u_{1,h \tau}$, where $u_{1, h \tau}:=\sum_{j=1}^LU_{j,h\tau} Y_j$, we denote the error  $e_1:=u^1_{h \tau}-u$.
 We write $e_1:=\varrho-\varepsilon$, where $\varrho:=u_{\mathcal{R}}^1-u$ and $\varepsilon:={u}_{\mathcal{R}}^1-u^1_{h \tau}:=\varepsilon_{0}+\epsilon \varepsilon_{1}$. Here, $\varepsilon_1:={u}_{\mathcal{R},1}-u_{1,h \tau}$ and ${u}_{\mathcal{R}}^1:={u}_{\mathcal{R},0}+\varepsilon {u}_{\mathcal{R},1}$,  where ${u}_{\mathcal{R},1}:=\sum_{j=1}^L{U}_{\mathcal{R},j}Y_j$. 
 \begin{lemma}
For $U_{j,h}^n \in V^n$, we have the following estimate of the elliptic reconstruction error.
\begin{align}
    \|\mathcal{R}^nU_{j,h}^n-U_{j,h}^n\| \leq C_{ER3}\|h_n^2 \mathfrak{R}_{1,j}^n\|+C_{ER4}\|h_n^{\frac{3}{2}}\mathfrak{J}_{1,j}^n\|_{\Sigma_n}.
\end{align} 
\end{lemma}
\begin{proof}
The result follows from a standard Aubin--Nitsche duality argument \cite{AO97, V96} along the same lines as \cite[Lemma 2.3]{LM06}, and the proof is omitted for brevity.
\end{proof}
\begin{lemma}\label{Lemma4}  For all $\phi\in V$, $U^n_{j,h}$, $j=1,2,\cdots, N$, satisfy
\begin{align}
    \int_D \widehat{\partial}U^n_{j,h} \phi + \int_DA^nU^n_{j,h} \phi-\int_{\Gamma_{\mathrm{R}}}(\alpha-\alpha_0) U^n_{j,h} \phi+ \int_{\Gamma_{\mathrm{R}}}\alpha_ju^n_{0,h} \phi =0.
\end{align}
\end{lemma}
\begin{proof}
In light of the projection operators and the BE scheme \eqref{BE2}, we obtain the desired result, and this concludes the proof. 
\end{proof}

Now, we obtain the following result using Lemmas \ref{Lemma1} and \ref{Lemma4} with $u^{1,n}_{h}:=u_{0,h}^n+\epsilon u_{1,h}^n$, where $u_{1,h}^n:=\sum_{j=1}^LU_{j,h}^n Y_j$.
\begin{lemma}\label{Lemma5} For every $\phi\in V$, the following result is true
 \begin{align}
    &\int_D  {\bar{\partial}}u^{1,n}_{h} \phi +\int_DA^nu^{1,n}_{h} \phi-\epsilon\int_{\Gamma_{\mathrm{R}}}(\alpha-\alpha_0)u_{1,h}^n \phi
    \\ & \qquad 
    - \int_D \mathfrak{f}^n \phi - \int_{\Gamma_{\mathrm{R}}}\mathfrak{g}^n \phi -\frac{1}{\tau_n} \int_D 
    (P_0^n - I) u_{h}^{1,n-1}\phi=0.\label{DF3}
\end{align}
\end{lemma}
\begin{proof}
The proof is straightforward in view of Lemmas \ref{Lemma1} and \ref{Lemma4}. \end{proof}

\begin{lemma}\label{Lemma6} The error equation for $\varrho$, $\forall \phi \in V$ is given by 
\begin{align}
 \int_{D} \dfrac{\partial \varrho}{\partial t} \phi+a(\varrho, \phi; \cdot)=&\int_{D} \dfrac{\partial \varepsilon}{\partial t} \phi + a(u_{\mathcal{R}}^1-u_{\mathcal{R}}^{1,n}, \phi; \cdot)
 + \int_D(\mathfrak{f}^n-f)\phi+\int_{ \Gamma_{\mathrm{R}}}(\mathfrak{g}^n-g)\phi \nonumber\\
& +\frac{1}{\tau_n} \int_D 
    (P_0^n - I) u_{h}^{1,n-1}\phi+ \epsilon\int_{ \Gamma_{\mathrm{R}}} (\alpha-\alpha_0)u_{1,h}^n  \phi,\label{Lemma2j}
 \end{align}
 where $u_{\mathcal{R}}^{1,n} = u_{\mathcal{R}}^1(t_n)$.
 \end{lemma}
 \begin{proof}
Using \eqref{ERec}, $\forall \phi \in  V$, Lemma \ref{Lemma5} can be written as
\begin{align}
    \int_D  {\bar{\partial}}u^{1,n}_{h} \phi + a(u_{\mathcal{R}}^{1,n},\phi;\cdot)-\epsilon\int_{\Gamma_{\mathrm{R}}}(\alpha-\alpha_0)u_{1,h}^n \phi- \int_D \mathfrak{f}^n \phi - \int_{\Gamma_{\mathrm{R}}}\mathfrak{g}^n \phi -\frac{1}{\tau_n} \int_D 
    (P_0^n - I) u_{h}^{1,n-1}\phi=0. \label{PWE2}
\end{align}
Using \eqref{wfu}, it follows that 
 \begin{align}
    \int_{D}\dfrac{\partial \varrho}{\partial t} \phi +a(\varrho, \phi; \cdot) 
    = \int_{D} \frac{\partial u_{\mathcal{R}}^1}{\partial t} \phi +a(u_{\mathcal{R}}^1, \phi; \cdot) - \int_{D} f \phi
    -\int_{ \Gamma_{\mathrm{R}}}g \phi.  \label{1Th1}
\end{align}
Hence, we obtain the desired result by subtracting \eqref{PWE2} from \eqref{1Th1}.
\end{proof}

We estimate $\varrho(t)$ in the following lemma.
\begin{lemma}\label{Lemma3j}
For $m =1,2, \cdots, N$, the following estimate holds:
\begin{align}
\left (\max_{[0,t_m]}\|\varrho(t)\|^2+2C_a \int_{0}^{t_m}\|\varrho(t)\|_V^2dt\right)^{\frac{1}{2}}\leq 2\|\varrho(0)\|+2\sqrt{2}({\sigma^2_{4,m}}+{\sigma^2_{5,m}})^{\frac{1}{2}}, \label{Lemma3.3j}
\end{align}
where
\begin{align}
\sigma_{4,m}^2:=2\left(\sum_{n=1}^m\tau_n(\zeta_{2,n,y}^\mathrm{S}+\zeta_{2,n}^\mathrm{T}+\zeta_{2,n}^{\mathrm{D}_1})\right)^2, \;\;\sigma_{5,m}^2:=\frac{1}{C_a}\sum_{n=1}^m{\tau_n}\left(\zeta_{2,n}^{\mathrm{D}_2}+\zeta_{2,n}^\mathrm{DM}+\zeta_{2,n,\mathbf{y}}^\mathrm{St}\right)^2. \label{sigma1,2}
\end{align}
Here,
$\zeta_{2,n,\mathbf{y}}^\mathrm{S}=\left(\zeta_{2,n}^{S_1}+\epsilon\sum_{j=1}^LY_j\zeta_{2,n}^{S_2}\right)$ and 
$\zeta_{2,n}^{S_1}=\zeta_{1,n}^\mathrm{S}$ is given in \eqref{1ThS}. Also,
\begin{align}
\zeta_{2,n}^{S_2}&:=  C_{1}\left\{\|\hat{h}_n^2 \bar{\partial} \mathfrak{R}_{1,j}^n\|+\|\hat{h}_n^{\frac{3}{2}} \bar{\partial} \mathfrak{J}_{1,j}^n\|_{\hat{\Sigma}_n}+\|\hat{h}_n^{\frac{3}{2}} \bar{\partial} \mathfrak{J}_{1,j}^n\|_{{\check{\Sigma}_n}/{\hat{\Sigma}_n}}\right\}, \label{1Th7} \\
\zeta_{2,n, \mathbf{y}}^\mathrm{T}&:=\|A^{n-1}u_{0,h}^{{n-1}}-A^{n}u_{0,h}^{{n}}\|+\epsilon \sum_{j=1}^L Y_j\|(A^{n-1}U_{j,h}^{{n-1}}-A^{n}U_{j,h}^{{n}})\|, \label{2ThT} \\
\zeta_{2,n,\mathbf{y}}^\mathrm{St}&:=\left (\frac{1}{\tau_n}\int_{t_{n-1}}^{t_n}\epsilon^2\|(\alpha-\alpha_0)u_{1,h}^n\|^2_{L^2(\Gamma_{\mathrm{R}})}dt \right)^{\frac{1}{2}}, \label{2ThSt}
\end{align}
and  $\zeta_{2,n}^{\mathrm{D}_1}=\zeta_{1,n}^{\mathrm{D}_1}$,  $\zeta_{2,n}^{\mathrm{D}_2}=\zeta_{1,n}^{\mathrm{D}_2}$, and $\zeta_{2,n}^\mathrm{DM}=\zeta_{1,n}^\mathrm{DM}$.
\end{lemma}
\begin{proof}
Substituting $\phi=\varrho$ in Lemma \ref{Lemma6}, $\forall \phi \in V$, we arrive at 
\begin{align}
 {\frac{1}{2}}\frac{d}{dt}\int_{D}\varrho ^2 + a(\varrho, \varrho; \cdot)=&\int_{D} \dfrac{\partial \varepsilon}{\partial t}\varrho  +a(u_{\mathcal{R}}^1- u_{\mathcal{R}}^{1,n},\varrho ; \cdot)+ 
 \int_D(\mathfrak{f}^n-f)\varrho+\int_{ \Gamma_{\mathrm{R}}}(\mathfrak{g}^n-g)\varrho \nonumber\\
 & +\frac{1}{\tau_n} \int_D 
    (P_0^n - I) u_{h}^{1,n-1}\varrho
    + \epsilon\int_{ \Gamma_{\mathrm{R}}}(\alpha-\alpha_0)u_{1,h}^n \varrho.
    \label{2Th1}
 \end{align}
Applying \eqref{coer} and after integration, the above equation reduces to 
\begin{align}
\frac{1}{2}\|\varrho(t_m)\|^2+&{C_a}\int_0^{t_m}  \|\varrho\|^2_V dt\leq
 \frac{1}{2}\|\varrho (0)\|^2+I_m,
\end{align}
where $\mathcal{I}_m:= \sum_{n=1}^m \left (\mathcal{I}_n^1+\mathcal{I}_n^2+\mathcal{I}_n^3+\mathcal{I}_n^4+\mathcal{I}_n^5\right)$ with
\begin{align}
\mathcal{I}_n^1:= &\int_{t_{n-1}}^{t_n}
\left|\left\langle\dfrac{\partial \varepsilon}{\partial t}, \varrho \right \rangle \right| dt, \nonumber\\
\mathcal{I}_n^2:=&\int_{t_{n-1}}^{t_n}|a(u_{\mathcal{R}}^1-u_{\mathcal{R}}^{1,n}, \varrho; \cdot)| dt, \nonumber\\
\mathcal{I}_n^3:=&\int_{t_{n-1}}^{t_n}\left (\left|\int_{D} (f^n-f) \varrho\right|
    +\left|\int_{ \Gamma_{\mathrm{R}}}\left(g^n-g\right) \varrho\right|\right)dt, \nonumber\\
\mathcal{I}_n^4:=&\int_{t_{n-1}}^{t_n} \left (\left|\int_D \left(P_0^n-I\right)\left (\frac{u_{h}^{1,n-1}}{\tau_n}+f^n\right) \varrho\right|+\left|\int_{\Gamma_{\mathrm{R}}} (P_{\frac{1}{2}}^n-I)g^n\varrho\right|\right)dt, \nonumber\\
\mathcal{I}_n^5:=& \epsilon\int_{t_{n-1}}^{t_n}\left|\int_{\Gamma_{\mathrm{R}}} (\alpha-\alpha_0) u_{1,h}^n \varrho\right|dt.  \nonumber
\end{align}
Setting $\|\varrho(t_m^*)\|^2=\max_{t \in [0, t_m]}\|\varrho(t)\|^2,  t_m^* \in [0,t_m]$, we have 
\begin{align}
\frac{1}{2}\|\varrho(t_m^*)\|^2+ C_a \int_0^{t_m} \|\varrho(t)\|_V^2 dt \leq \|\varrho (0)\|^2+2\mathcal{I}_m.   \label{TotBound}
\end{align}
Now, we try to bound each term separately.\\
\textbf{Bound on $\mathcal{I}_n^1 \colon$ (Spatial error estimate)}
We obtain the following error estimate for the spatial error 
\begin{align}
\sum_{n=1}^m\mathcal{I}_n^1 \leq \|\varrho(t_m^*)\|\sum_{n=1}^m \zeta_{2,n,\mathbf{y}}^\mathrm{S}\tau_n, 
\end{align}
where $\zeta_{2,n,\mathbf{y}}^\mathrm{S}$ is given in Lemma \ref{Lemma3j}. 
The proof follows the same line of argument as in \cite[Lemma 2.3]{LM06}, relying on the standard Aubin--Nitsche duality technique \cite{AO97, V96}; hence, it is omitted for brevity. \\
\textbf{Bound on $\mathcal{I}_n^2 \colon$(Time error estimate)}
The use of \eqref{ERec} shows that
\begin{align}
\mathcal{I}_n^2 
    \leq& \int_{t_{n-1}}^{t_n}l_{n-1}(t)\|A^{n-1}u_{0,h}^{{n-1}}-A^{n}u_{0,h}^{{n}}\| \|\varrho(t)\|dt \nonumber\\
    + &\epsilon \sum_{j=1}^L Y_j\int_{t_{n-1}}^{t_n}l_{n-1}(t)\|(A^{n-1}U_{j,h}^{{n-1}}-A^{n}U_{j,h}^{{n}})\| \|\varrho(t)\|dt.\nonumber
\end{align}
Thus,
\begin{align}
\sum_{n=1}^m\mathcal{I}_n^2 \leq \|\varrho(t_m^*)\|\sum_{n=1}^m \zeta_{2,n,\mathbf{y}}^\mathrm{T} \tau_n,
\end{align}
where $\zeta_{2,n, \mathbf{y}}^\mathrm{T}$ is given by \eqref{2ThT}.\\
\textbf{Bounds on $\mathcal{I}_n^3$, $\mathcal{I}_n^4$ and $\mathcal{I}_n^5 \colon$}
 Similarly to the bounds on $I_n^3$,  $I_n^4$, and  $I_n^5$ in the estimates of order $O(\epsilon)$, we obtain
\begin{align}
    \sum_{n=1}^m \mathcal{I}_n^3 \leq \|\varrho(t_m^*)\|\sum_{n=1}^m \zeta_{2,n}^{\mathrm{D}_1}\tau_n +\sum_{n=1}^m \zeta_{2,n}^{\mathrm{D}_2}\tau_n^{\frac{1}{2}} \left (\int_{t_{n-1}}^{t_n}\|\varrho\|_V^2dt\right)^{\frac{1}{2}} ,
\end{align}
where $\zeta_{2,n}^{\mathrm{D}_1}=\zeta_{1,n}^{\mathrm{D}_1}$  and $\zeta_{2,n}^{\mathrm{D}_2}=\zeta_{1,n}^{\mathrm{D}_2}$ and are given by \eqref{dataf1} and \eqref{datag1}, respectively. Moreover,
\begin{align}
\sum_{n=1}^{m}\mathcal{I}_n^4 \leq \sum_{n=1}^{m}\tau_n^{\frac{1}{2}}\zeta_{2,n}^\mathrm{DM}\left (\int_{t_{n-1}}^{t_n}\|\varrho\|_V^2 dt\right)^{\frac{1}{2}},
\end{align}
where $\zeta_{2,n}^\mathrm{DM}=\zeta_{1,n}^\mathrm{DM}$ and is given by \eqref{1ThDM}. Also,
\begin{align}
\sum_{n=1}^m \mathcal{I}_n^5\leq \sum_{n=1}^m \tau_n^{\frac{1}{2}}\zeta_{2,n,\mathbf{y}}^\mathrm{St}\left (\int_{t_{n-1}}^{t_n}\|\varrho\|_V^2dt\right)^{\frac{1}{2}}, \nonumber
\end{align}
where $\zeta_{2,n,\mathbf{y}}^\mathrm{St}$ is given in \eqref{2ThSt}.
Finally, using all the bounds in \eqref{TotBound}, it follows that
\begin{align}
\frac{1}{2}\|\varrho(t_m^*)\|^2+C_a\int_0^{t_m}\|\varrho\|_V^2 dt\leq & \|\varrho(0)\|^2+2\|\varrho(t_m^*)\| \sum_{n=1}^m\left (\zeta_{2,n,\mathbf{y}}^\mathrm{S}+\zeta_{2,n,\mathbf{y}}^\mathrm{T}+\zeta_{2,n}^{\mathrm{D}_1}\right)\tau_n\\
&+2\sum_{n=1}^m\left(\int_{t_{n-1}}^{t_{n}}\|\varrho\|_V^2dt\right)^\frac{1}{2}\left(\zeta_{2,n}^{\mathrm{D}_2}+\zeta_{2,n}^\mathrm{DM}+\zeta_{2,n,\mathbf{y}}^\mathrm{St}\right)\tau_n^{\frac{1}{2}}. \nonumber
\end{align}
Using a similar kick-back argument as in the end of the proof of Lemma \ref{Lemma3}, we obtain the estimate for the parabolic error. This completes the rest of the proof. 
\end{proof}

\begin{thm}
Let $u$ be the solution of problem \eqref{wfu} and $u_{0,h \tau}$ be defined as in \eqref{linapx1}. 
Then, there exists a constant $C_2 > 0$ depending only on the trace, Poincar\'e constants, and the mesh aspect
ratio such that
\begin{align}
\left(E\left[\max_{[0,t_m]}\|e_1(t)\|^2\right]\right)^{\frac{1}{2}} \leq 2\left (2\sqrt{2}E\Big[\|\varrho(0)\|^2\Big]+(\zeta_{1,n}^\mathrm{R})^2+\epsilon^2(\zeta_{2,n}^\mathrm{R})^2+8({\sigma^2_{4,m}}+ \sigma^2_{6,m}\right)^{\frac{1}{2}}, \label{Thm2}
\end{align}
where $ \sigma^2_{4,m}$ and $\zeta_{1,n}^\mathrm{R}$ are given in \eqref{sigma1,2}, and \eqref{Elp1}, respectively. Moreover,
\begin{align}
(\zeta_{2,n}^\mathrm{R})^2&:= C_{5}\sum_{j=1}^L(\|\hat{h}_n^2 \mathfrak{R}_{1,j}^n\|+\|\hat{h}_n^{\frac{3}{2}}  \mathfrak{J}_{1,j}^n\|_{{\Sigma}_n}), \label{ElpRec2} \\
 \sigma^2_{6,m}&:=\frac{3}{C_a}\left (\sum_{n=1}^{m}\tau_n(\zeta_{2,n}^{\mathrm{D}_2})^2+\sum_{n=1}^{m}\tau_n (\zeta_{2,n}^\mathrm{DM})^2 +\sum_{n=1}^{m}\tau_n(\zeta_{2,n}^\mathrm{St})^2\right),
\end{align}
with
\begin{align}
(\zeta_{2,n}^\mathrm{St})^2:=\epsilon^4\left [\int_{ \Gamma_{\mathrm{R}}}\sum_{i=1}^{L}|\alpha_iU_{i,h\tau}|^2E\Big[Y_i^4\Big]+\int_{\Gamma_{\mathrm{R}}}\sum_{i\neq j=1}^L(\alpha^2_i|U_{j,h}|^2+2 \alpha_i \alpha_jU_{i,h}U_{j,h})\right],\label{zeta-St-sec}
\end{align}
and  $\zeta_{2,n}^{\mathrm{D}_1}=\zeta_{1,n}^{\mathrm{D}_1}$,  $\zeta_{2,n}^{\mathrm{D}_2}=\zeta_{1,n}^{\mathrm{D}_2}$, and $\zeta_{2,n}^\mathrm{DM}=\zeta_{1,n}^\mathrm{DM}$. Here, the positive constant $C_5$ depends on the mesh aspect ratio, and $C_a$ is the coercivity constant.
\end{thm}
\begin{proof}
For a.e. $t \in (0,T)$, we arrive at 

\begin{align}
E\Big[\max_{t \in [0,t_m]}\|e_1(t)\|^2\Big]\leq 2E\Big[\max_{n \in [0;m]}\|\varepsilon(t)\|^2+ \max_{n \in [0;m]}\|\varrho(t)\|^2\Big]. \label{2Th31} 
\end{align}
Now, each term on the right hand side can be bounded separately. 
Taking into account the bounds on $\varepsilon_0$ and $\varepsilon_1$, we obtain
\begin{align}
     \max_{t \in [0,t_m]}\|\varepsilon(t)\|^2\leq 2\left (\max_{n \in [0;m]} (\zeta_{1,n}^\mathrm{R})^2+\epsilon^2 \max_{n \in [0;m]} (\zeta_{2,n, \mathbf{y}}^\mathrm{R})^2 \right),
\end{align}
where $(\zeta_{1,n}^\mathrm{R})^2$ is given by \eqref{Elp1} and $\zeta_{2,n,\mathbf{y}}^\mathrm{R}:=C_{5}\sum_{j=1}^L(\|h_n^2 \mathfrak{R}_{1,j}^nY_j\|+\|{h}_n^{\frac{3}{2}}  \mathfrak{J}_{1,j}^nY_j\|_{{\Sigma}_n})$.

Applying the expectation on both sides, in view of \eqref{EY2}, we obtain
\begin{align}
E\Big[\max_{t \in [0,t_m]}\|\varepsilon(t)\|^2\Big] \leq2\left (\max_{n \in [0;m]} (\zeta_{1,n}^\mathrm{R})^2+\epsilon^2\max_{n \in [0;m]}(\zeta_{2,n}^\mathrm{R})^2 \right),\label{2Th29}
\end{align}
where $(\zeta_{2,n}^\mathrm{R})^2$ is given in \eqref{ElpRec2}. 
Now, from the Lemma \ref{Lemma3j}, it follows that
\begin{align}
\max_{t \in [0,t_m]}\|\varrho(t)\|^2 &\leq 
8\|\varrho(0)\|^2+16({\sigma^2_{4,m}}+{\sigma^2_{5,m}}).
 \nonumber
\end{align}
Now applying the expectations on both sides, we arrive at
\begin{align}
E\Big[\max_{n \in [0,m]}\|\varrho\|^2\Big]& \leq 8E\Big[\|\varrho(0)\|^2\Big]+16\left (\sigma_{4,m}^2+E\Big[\sigma_{5,m}^2\Big]\right), \label{2Th32}
\end{align}
where
\begin{align}
E\Big[\sigma_{5,m}^2\Big]=\frac{3}{C_a}\left (\sum_{n=1}^{m}\tau_n(\zeta_{2,n}^{\mathrm{D}_2})^2+\sum_{n=1}^{m}\tau_n(\zeta_{2,n}^\mathrm{DM})^2+\sum_{n=1}^{m}\tau_nE(\zeta_{2,n, \mathbf{y}}^\mathrm{St})^2\right). \label{2Th34}
\end{align}
Now, substituting \eqref{2Th29} and \eqref{2Th32} in \eqref{2Th31}, and employing \eqref{EY2} and
\begin{equation}
E[Y_{i}Y_{j}Y_{k}Y_{l}]=%
\begin{cases}
E[Y_{i}^{4}],\;\;i=j=k=l,\\
1,\;\;\;\text{if the indices are pairwise equal,}\label{EY4}\\
0,\;\;\text{otherwise},
\end{cases}
\end{equation} we arrive at  the desired result. This concludes the rest of the proof.
\end{proof}

\noindent
\textbf{Remark 1.}
The obtained \emph{a posteriori} error estimators are robust in the sense that the constants appearing in the bounds are independent of $\epsilon$, as well as of the mesh size $h$ and the time-step $\tau$. Moreover, the present analysis generalizes the results of \cite{LM06}, originally developed for deterministic parabolic problems with Dirichlet boundary conditions, to parabolic problems with small uncertainties in Robin-type boundary conditions.
\hfill\qedsymbol  
\section{Numerical investigations\label{sec:5}} 
\begin{figure}[H]
\centering
\begin{minipage}{0.6\textwidth}
In this section, numerical experiments are conducted to validate 
the derived \emph{a posteriori} error estimates. Tests 
are carried out at $T=1$ with $k=I$ on the domain $D=(0,1)^{2}$, 
comprising the Robin boundary 
$\Gamma_{\mathrm{R}}=\Gamma_{\mathrm{R}_{1}}\cup\Gamma_{\mathrm{R}_{2}}\cup\Gamma_{\mathrm{R}_{3}}$, 
and the Dirichlet boundary $\Gamma_{\mathrm{D}}$ (please refer to 
Figure~\ref{Domain}). The spatial approximations 
are treated with $\mathbb{P}_{1}$ finite elements. 
Setting $\Gamma_{\mathrm{D}}\neq\emptyset$, we define 
\begin{equation}
\text{Error}:=E\left[\max_{t \in [0, T]} \| e_{1}(t,\cdot,\cdot) \|^2\right]^{\frac{1}{2}}.  \label{numerr:ch:4}
\end{equation}
To this end, we consider a numerical example \cite{G19,SRM25} with the exact solution given by
\begin{equation}
u_{0}(t,x_{1},x_{2})=\text{sin}\left({5\pi t}\right)\text{sin}%
\left(\frac{\pi x_{1}}{2}\right)\text{sin}\left(\frac{\pi x_{2}}{2}%
\right). \label{ex1}%
\end{equation}
\end{minipage}
    \hfill
    \begin{minipage}{0.35\textwidth}
        \centering
        \begin{tikzpicture} [{very thick}, minimum size=3mm]
            \draw (0,0) rectangle (3,3);
            \draw[very thick] node at (2,2.5) {$D$};
            \draw[very thick] node at (3.5,1.5) {$\Gamma_{\mathrm{R}_2}$};
            \draw[very thick] node at (1.5,-.5) {$\Gamma_{\mathrm{R}_1}$};
            \draw[very thick] node at (1.5,3.5) {$\Gamma_{\mathrm{R}_3}$};
            \draw[very thick] node at (-.5,1.5) {$\Gamma_{\mathrm{D}}$};
        \end{tikzpicture}
        \caption{\sl \small Domain $D$ with Dirichlet ($\Gamma_{\mathrm{D}}$) and Robin ($\Gamma_{\mathrm{R}_{1}},\Gamma_{\mathrm{R}_{2}},\Gamma_{\mathrm{R}_{3}}$) boundaries.}
        \label{Domain}
    \end{minipage}
\end{figure}
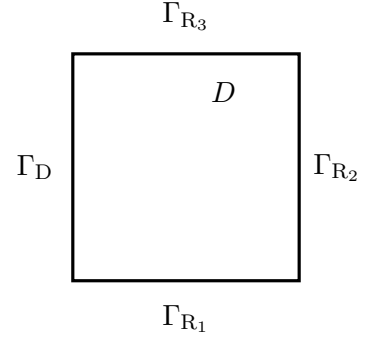
The random coefficient $\alpha$ is given by
\begin{equation*}
\alpha(\textbf{x}, \mathbf{Y}(\omega))=\alpha_0(\textbf{x})+\epsilon\sum_{j=1}^3\alpha_j(\mathbf{x})Y_j(\omega),
\end{equation*}
with $ \alpha_0(\mathbf{x})=1$, $\alpha_j=\mathcal{X}_{\Gamma_{\mathrm{R}_j}}$, $\mathcal{X}$ is the indicator function and $Y_j$'s are independent random variables that follow uniform distribution in $[-\sqrt{3}, \sqrt{3}]$. Upon substituting $u_{0}$, $\alpha_{0}$ and $k=I$ into \eqref{det}, we obtain $f$, $g$, and the initial condition $u_{\text{init}}$. 

All numerical results are computed using {FreeFEM} v4.6 and {MATLAB} R2025a. We first consider the deterministic case.

\vspace{.1cm}
\subsection*{Deterministic case}
In this study, with $\epsilon = 0$, we examine the behavior of the estimators
\[
\begin{aligned}
\zeta_{1}^\mathrm{S} &= \Big(\sum_{n=1}^{N}\tau_n(\zeta_{1,n}^\mathrm{S})^2\Big)^{1/2}, \quad
\zeta_{1}^\mathrm{R} = \Big(\max_{n\in[0,N]}(\zeta_{1,n}^\mathrm{R})^2\Big)^{1/2},\\
\zeta_{1}^\mathrm{T} &= \Big(\sum_{n=1}^{N}\tau_n(\zeta_{1,n}^\mathrm{T})^2\Big)^{1/2}, \quad
\zeta_{1}^\mathrm{D} = \Big(\sum_{n=1}^{N}\tau_n(\zeta_{1,n}^{\mathrm{D}_1})^2\Big)^{1/2}
+ \Big(\sum_{n=1}^{N}\tau_n(\zeta_{1,n}^{\mathrm{D}_2})^2\Big)^{1/2},
\end{aligned}
\]
where $\zeta_{1,n}^\mathrm{S}$, $\zeta_{1,n}^\mathrm{R}$, $\zeta_{1,n}^\mathrm{T}$, $\zeta_{1,n}^{\mathrm{D}_1}$, and $\zeta_{1,n}^{\mathrm{D}_2}$ are defined in \eqref{1ThS}, \eqref{1ThT}, \eqref{Elp1}, \eqref{dataf1}, and \eqref{datag1}, respectively. Using uniform meshes $h = 0.25, 0.125, 0.0625$ and constant time steps $\tau = 0.01, 0.0025, 0.000625$, we study the numerical convergence of the error and estimators under the coupling $\tau = c h^{2}$ with $c = 0.16$. All constants appearing in the estimators are set to $1$.

 Figure~\ref{BE_Err_Est} shows the error and estimators $\zeta_1^\mathrm{S}$ (space), $\zeta_1^\mathrm{R}$ (elliptic reconstruction), $\zeta_1^\mathrm{T}$ (time), and $\zeta_1^\mathrm{D}$ (data) versus the number of elements on a log--log scale. All curves decrease monotonically, approximately parallel to the reference triangle with slope $-1$, exhibiting second-order convergence in $h$. The space estimator $\zeta_1^\mathrm{S}$ dominates, $\zeta_1^\mathrm{R}$ is smaller but decays similarly, and $\zeta_1^\mathrm{T}$ and $\zeta_1^\mathrm{D}$ are comparable. The true error lies below the estimators while following the same slope, indicating optimal asymptotic behavior.

\begin{figure}[H]
\begin{center}
\includegraphics[width=0.99\textwidth]{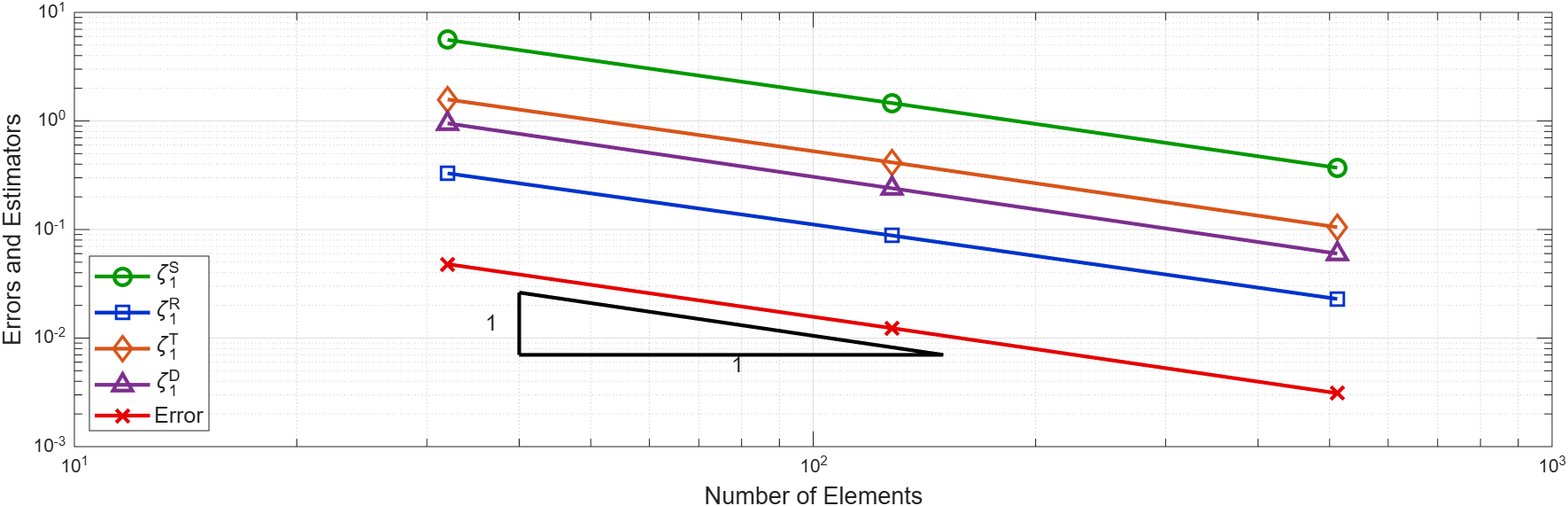}
\end{center}
\caption{\footnotesize\textsl{The estimators $\zeta_1^\mathrm{S}$ (green), $\zeta_1^\mathrm{R}$ (blue), $\zeta_1^\mathrm{T}$ (orange), and $\zeta_1^\mathrm{D}$ (purple) are plotted against the number of elements on a log--log scale, along with the true error (red). The black reference triangle with slope $-1$ indicates second-order spatial convergence in 2D. The coupling $\tau = c h^{2}$ with $c=0.16$ ensures optimal decay of error and estimators.
}}
\label{BE_Err_Est}%
\end{figure}
\subsection*{Random case}
We now consider the random case, using a sample size of 100 and the standard MC method to estimate the true error in \eqref{numerr:ch:4}. The reference solution is computed with $h_{\text{ref}} = 1/256$ and $\tau_{\text{ref}} = 2^{-9}$. We focus on the behavior of the error and the stochastic estimator $\zeta_1^\mathrm{St}$ to illustrate the first-order convergence in $\epsilon$, where
\[
(\zeta_{1}^\mathrm{St})^{2} = \sum_{n=1}^{N} \int_{t_{n-1}}^{t_n} (\zeta_{1,n}^\mathrm{St})^2 \, dt,
\]
with $\zeta_{1,n}^\mathrm{St}$ defined in \eqref{1eta-St}.

The first-order results for the error corresponding to $h=0.25$, $0.125$, and $0.0625$ for decreasing $\epsilon$ are shown in Figure~\ref{Stoch_err_first}. We observe that an appropriate balance between the parameters $h$, $\tau$, and $\epsilon$ is required to clearly observe the convergence behavior. Nevertheless, the ratio of errors between successive runs increases as $\epsilon$ decreases, indicating an improvement in convergence behavior.
\begin{figure}[H]
    \centering
    \begin{subfigure}{0.47\textwidth}
        \centering
        \includegraphics[width=\textwidth]{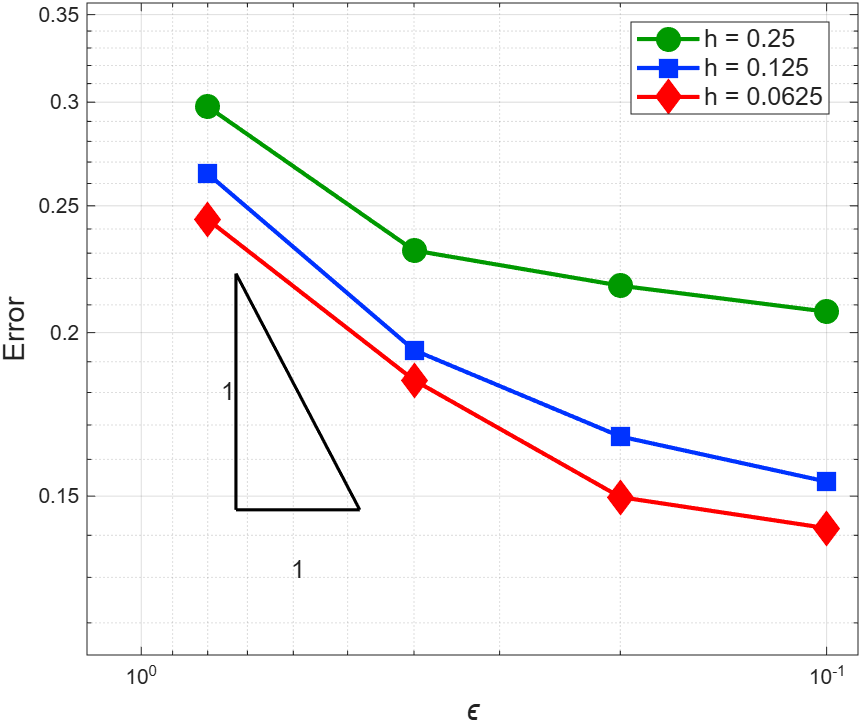}
        \caption{\sl \footnotesize Illustration of error behavior w.r.t. $\epsilon$ for  fixed 
        $h=0.25$, $0.125$, and $0.0625$ under the coupling $h\propto\tau^{2}$. }
        \label{Stoch_err_first}
    \end{subfigure}
    \hfill
    \begin{subfigure}{0.5\textwidth}
        \centering
\includegraphics[width=\textwidth]{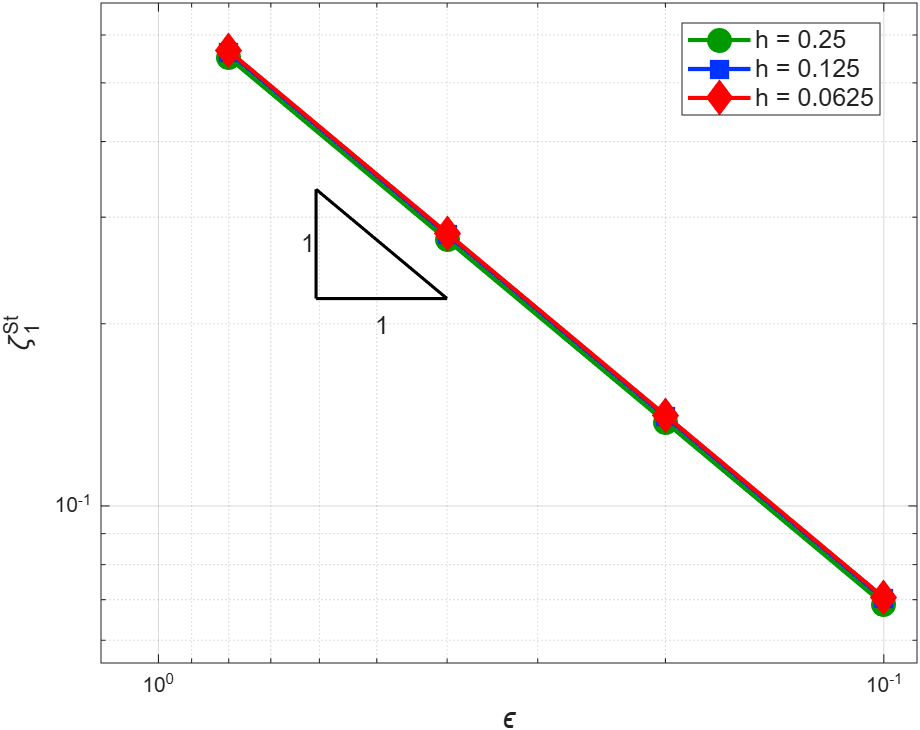}
\caption{\sl \footnotesize Optimal behavior of stochastic estimator $\zeta_{1}^{\mathrm{st}}$   w.r.t. $\epsilon$ for  fixed 
        $h=0.25$, $0.125$, and $0.0625$ under the coupling $h\propto\tau^{2}$.}
        \label{Stoch_esti_first}
    \end{subfigure}
\caption{\footnotesize \textsl{The true error and stochastic estimator $\zeta_1^\mathrm{St}$ 
are plotted against the $\epsilon$ on a log--log scale, 
for $h=0.25$ (green), $0.125$ (blue), and $0.0625$ (red). The black reference triangle with slope $1$ indicates expected first-order convergence in $\epsilon$. 
}}
\label{Stoch_err_esti_first}
\end{figure}
Figure~\ref{Stoch_esti_first} illustrates the behavior of the estimator $\zeta_{1}^{\mathrm{st}}$ with respect to the perturbation parameter $\epsilon$ for fixed mesh sizes $h=0.25$, $0.125$, and $0.0625$. As $\epsilon$ decreases, the estimator exhibits an approximately linear decay in the log--log scale, as indicated by the reference triangle of slope one. This confirms the expected first-order convergence with respect to $\epsilon$. Moreover, the curves for different mesh sizes follow the same trend, with smaller magnitudes observed for finer meshes, demonstrating consistent behavior with mesh refinement.
\begin{figure}[H]
    \centering
    \begin{subfigure}{0.48\textwidth}
        \centering
        \includegraphics[width=\textwidth]{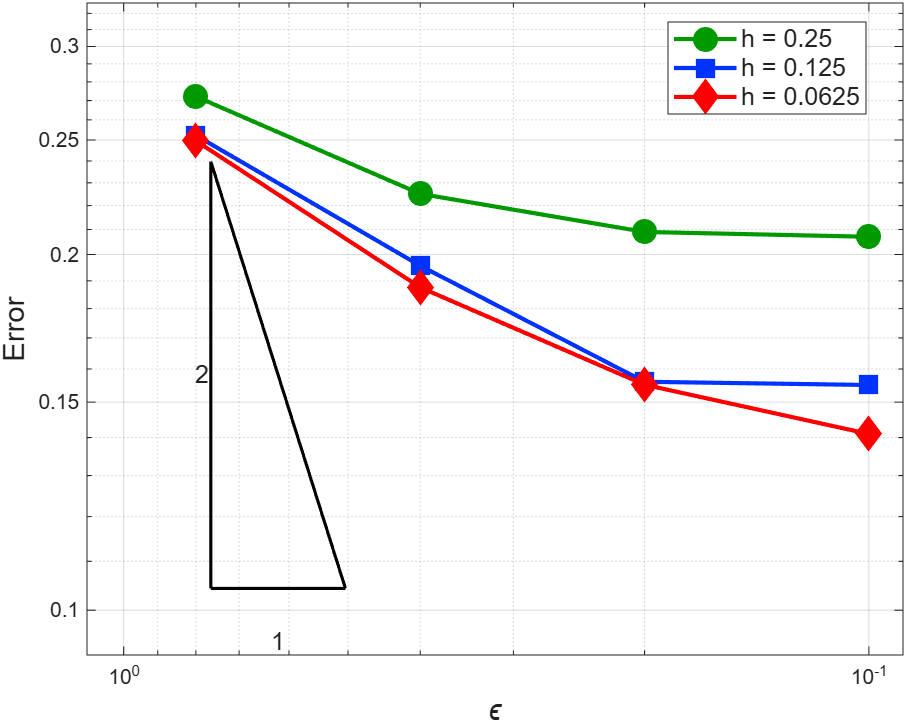}
        \caption{Illustration of error behavior with respect to $\epsilon$ for  fixed 
        $h=0.25$, $0.125$, and $0.0625$ under the coupling $h\propto\tau^{2}$.}
        \label{Stoch_err_sec}
    \end{subfigure}
    \hfill
    \begin{subfigure}{0.48\textwidth}
        \centering
\includegraphics[width=\textwidth]{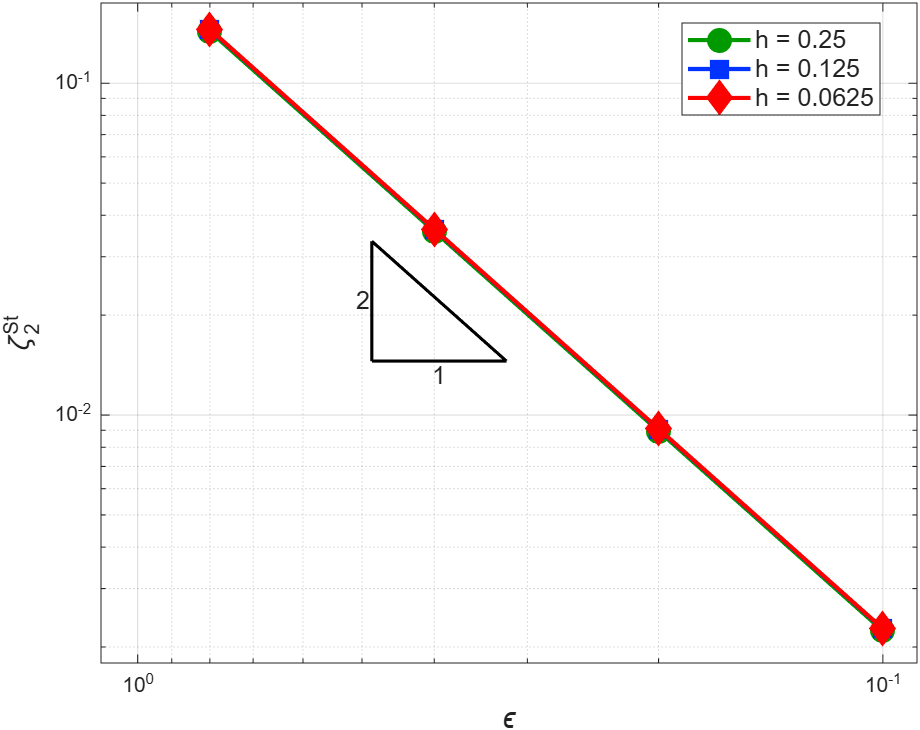}
\caption{\sl \small Optimal behavior of the stochastic estimator $\zeta_{2}^{\mathrm{st}}$  with respect to $\epsilon$ for  fixed 
        $h=0.25$, $0.125$, and $0.0625$ under the coupling $h\propto\tau^{2}$.}
        \label{Stoch_esti_sec}
    \end{subfigure}
\caption{\footnotesize \textsl{The true error and estimator $\zeta_2^\mathrm{St}$ 
are plotted against the $\epsilon$ on a log--log scale, 
for $h=0.25$ (green), $0.125$ (blue), and $0.0625$ (red). The black reference triangle with slope $2$ indicates second-order convergence in $\epsilon$. It shows the optimal behavior of $\zeta_2^\mathrm{St}$  with respect to $\epsilon$.}}
\label{Stoch_err_esti_sec}
\end{figure}
To illustrate the behavior of $O(\epsilon^{2})$, we also consider the stochastic estimator
\begin{equation*}
(\zeta_{2}^{\mathrm{st}})^{2}
=
\sum_{n=1}^{N}\int_{t_{n-1}}^{t_{n}}
(\zeta_{2,n}^{\mathrm{st}})^{2}\,dt ,
\end{equation*}
where $\zeta_{2,n}^{\mathrm{st}}$ is defined in \eqref{zeta-St-sec}.

The corresponding second-order results in $\epsilon$ are presented in Figure~\ref{Stoch_err_esti_sec}, which shows the behavior of the error and the estimator $\zeta_{2}^{\mathrm{st}}$ for fixed mesh sizes $h=0.25$, $0.125$, and $0.0625$. As $\epsilon$ decreases, the estimator exhibits a quadratic decay in the log--log scale, confirming the expected second-order convergence with respect to $\epsilon$. However, as previously observed, an appropriate balance among the parameters $h$, $\tau$, and $\epsilon$ is required to clearly observe the convergence behavior of the error. To conclude, the \emph{a posteriori}
error estimators obtained in this article are robust with respect to $\epsilon$, $h$ and $h$.

\section{Conclusions}\label{sec:6}
This study presents residual-based \emph{a posteriori} error estimates in the $L^2_P(\Omega; L^{\infty}(0, T;L^2(D)))$-norm for a parabolic PDE with a small uncertainty in a Robin boundary condition. The discretization of the problem is based on $\mathbb{P}_1$ finite elements in space and the BE scheme in time, with the perturbation technique applied to address randomness. Estimates of the order $O(\epsilon)$ and $O(\epsilon^2)$ were obtained, meticulously isolating contributions from space, time, and stochastic components. In order to obtain the estimates, a parametric elliptic reconstruction operator is introduced. 
The use of the elliptic reconstruction technique offers a new and promising approach to achieving this, expanding its applicability to a wider range of mathematical models. Numerical experiments demonstrate the effectiveness and robustness of the estimators.  

\noindent
\subsection*{Acknowledgement.} 
The second author gratefully acknowledges the financial support from NBHM, India [Grant Number No.02011/27/2025/NBHM(R.P)/R\&D II/9829]. 
The fourth author was partially supported by
the European Union (ERC, StochMan, 101088589).

\bibliography{SPPDE.bib}

@STRING{CMAME = {Comput. Methods Appl. Mech. Engng}}

@STRING{IJHMT = {Int. J. Heat Mass Transfer}}

@STRING{JCAM = {J. Comp. Appl. Math.}}

@STRING{JCP = {J. Comput. Phys.}}

@STRING{SIAM = {SIAM J. Appl. Math.}}

@STRING{IMAJNA = {IMA J. Numer. Anal.}}

@STRING{AN = {Acta Numer.}}

@STRING{IJHMT ={Int. J. Heat Mass Trans.}}

@STRING{JCAM = {J. Comput. Appl. Math.}}

@STRING{NMPDE = {Numer. Methods Partial Differ. Equ.}}

@STRING{IJNME = {Int. J. Numer. Meth. Eng.}}

@STRING{SINUM = {SIAM J. Numer. Anal.}}

@STRING{SISC={SIAM J. Sci. Comput. }}

@STRING{MC = {Math. Comp.}}

@STRING{JOMP = {J. Sci. Comput. }}

@STRING{ACME={Arch. Comput. Methods Eng.}}

@STRING{IJNME={Int. J. Numer. Methods Eng.}}

@STRING{IJNME={Internat. J. Numer. Methods Engrg.}}

@STRING{EMMNA={ESAIM Math. Model. Numer. Anal.}}

@STRING{IJHMT={Int. J. Heat Mass Transf.}}

@STRING{EJP={Eur. J. Phys.}}

@STRING{CP={Contemp. Phys.}}

@STRING{AMBP={Ann. Math. Blaise Pascal.}}

@article{KPSZ23,
  title={Efficient adaptive stochastic collocation strategies for advection--diffusion problems with uncertain inputs},
  author={Kent, Benjamin M and Powell, Catherine E and Silvester, David J and Zimo{\'n}, Ma{\l}gorzata J},
  journal=JOMP,
  volume={96},
  number={3},
  pages={64},
  year={2023},
  publisher={Springer}
}

@article{SRM25,
  title={A Posteriori error estimates for the {C}rank--{N}icolson method: {A}pplication to Parabolic Partial Differential Equations Subject to a {R}obin Boundary Condition with Small Randomness},
  author={Shravani, N and Reddy, G. M. M. and Vynnycky, M},
  journal=JOMP,
  volume={104},
  number={1},
  year={2025},
  publisher={Springer US New York}
}

@article{SZ90,
  title={Finite element interpolation of nonsmooth functions satisfying boundary conditions},
  author={Scott, L Ridgway and Zhang, Shangyou},
  journal=MC,
  volume={54},
  number={190},
  pages={483--493},
  year={1990}
}

@article{MP22,
  title={A posteriori error estimation and space-time adaptivity for a linear stochastic PDE with additive noise},
  author={Majee, A. K. and Prohl, A.},
  journal= IMAJNA,
  volume={42},
  number={2},
  pages={1526--1567},
  year={2022},
  publisher={Oxford University Press}
}

@incollection{NV21,
  title={A posteriori error estimation for the stochastic collocation finite element approximation of the heat equation with random coefficients},
  author={Nobile, F. and Vidli{\v{c}}kov{\'a}, E.},
  booktitle={Sparse Grids and Applications-Munich 2018},
  pages={127--159},
  year={2021},
  publisher={Springer}
}

@article{CPB19,
  title={Efficient adaptive multilevel stochastic {G}alerkin approximation using implicit a posteriori error estimation},
  author={Crowder, A. J. and Powell, C. E. and Bespalov, A.},
  journal=SISC,
  volume={41},
  number={3},
  pages={A1681--A1705},
  year={2019},
  publisher={SIAM}
}

@article{BPRR19,
  title={Convergence of adaptive stochastic {G}alerkin {F}{E}{M}},
  author={Bespalov, A. and Praetorius, D. and Rocchi, L. and Ruggeri, M.},
  journal=SINUM,
  volume={57},
  number={5},
  pages={2359--2382},
  year={2019},
  publisher={SIAM}
}

@article{EGSZ14,
  title={Adaptive stochastic {G}alerkin {FEM}},
  author={Eigel, M. and Gittelson, C. J. and Schwab, C. and Zander, E.},
  journal=CMAME,
  volume={270},
  pages={247--269},
  year={2014},
  publisher={Elsevier}
}

@article{GSRJ16,
  title={A posteriori error analysis of two-step backward differentiation formula finite element approximation for parabolic interface problems},
  author={Gupta, J. S. and Sinha, R. K. and Reddy, G. M. M. and Jain, J.},
  journal=JOMP,
  volume={69},
  number={1},
  pages={406--429},
  year={2016},
  publisher={Springer}
}

@inproceedings{V96,
  title={A review of a posteriori error estimation},
  author={Verf{\"u}rth, R.},
  booktitle={Adaptive Mesh-Refinement Techniques, Wiley \& Teubner},
  year={1996},
  organization={Citeseer}
}

@article{KL13,
  title={Maximum norm a posteriori error estimation for parabolic problems using elliptic reconstructions},
  author={Kopteva, N. and Linss, T.},
  journal=SINUM,
  volume={51},
  number={3},
  pages={1494--1524},
  year={2013},
  publisher={SIAM}
}

@article{D12,
title = {Newton’s law of cooling and its interpretation},
journal = IJHMT,
volume = {55},
number = {21},
pages = {5397-5402},
year = {2012},
issn = {0017-9310},
doi = {https://doi.org/10.1016/j.ijheatmasstransfer.2012.03.035},
url = {https://www.sciencedirect.com/science/article/pii/S0017931012001846},
author = {M. I. Davidzon},
}

@article {GA98,
    AUTHOR = {Gustafson, K. and Abe, T.},
     TITLE = {The third boundary condition---was it {R}obin's?},
   JOURNAL = {Math. Intelligencer},
  FJOURNAL = {The Mathematical Intelligencer},
    VOLUME = {20},
      YEAR = {1998},
    NUMBER = {1},
     PAGES = {63--71},
      ISSN = {0343-6993},
   MRCLASS = {01A55 (35-03)},
  MRNUMBER = {1601764},
MRREVIEWER = {Thomas Archibald},
       DOI = {10.1007/BF03024402},
       URL = {https://doi.org/10.1007/BF03024402},
}

@article{LM06,
  title={Elliptic reconstruction and a posteriori error estimates for fully discrete linear parabolic problems},
  author={Lakkis, O. and Makridakis, C.},
  journal=MC,
  volume={75},
  number={256},
  pages={1627--1658},
  year={2006}
}

@article{BKM12,
  title={A posteriori error control for fully discrete Crank--Nicolson schemes},
  author={B{\"a}nsch, E. and Karakatsani, F. and Makridakis, C.},
  journal=SINUM,
  volume={50},
  number={6},
  pages={2845--2872},
  year={2012},
  publisher={SIAM}
}

@article{GNP16,
  title={A posteriori error estimation for elliptic partial differential equations with small uncertainties},
  author={Guignard, D. and Nobile, F. and Picasso, M.},
  journal=NMPDE,
  volume={32},
  number={1},
  pages={175--212},
  year={2016},
  publisher={Wiley Online Library}
}

@article{GWZ14,
  title={Stochastic finite element methods for partial differential equations with random input data},
  author={Gunzburger, M. D. and Webster, C. G. and Zhang, G.},
  journal=AN,
  volume={23},
  pages={521--650},
  year={2014},
  publisher={Cambridge University Press}
}

@article{MN03,
  title={Elliptic reconstruction and a posteriori error estimates for parabolic problems},
  author={Makridakis, C. and Nochetto, R. H.},
  journal=SINUM,
  volume={41},
  number={4},
  pages={1585--1594},
  year={2003},
  publisher={SIAM}
}

@article{G19,
  title={Partial differential equations with random input data: A perturbation approach},
  author={Guignard, D.},
  journal=ACME,
  volume={26},
  number={5},
  pages={1313--1377},
  year={2019},
  publisher={Springer}
}

@article{ZL04,
  title={An efficient, high-order perturbation approach for flow in random porous media via {K}arhunen--{L}o\`eve and polynomial expansions},
  author={Zhang, D. and Lu, Z.},
  journal=JCP,
  volume={194},
  number={2},
  pages={773--794},
  year={2004},
  publisher={Elsevier}
}

@article{AO97,
  title={A posteriori error estimation in finite element analysis},
  author={Ainsworth, M. and Oden, J. T.},
  journal=CMAME,
  volume={142},
  number={1-2},
  pages={1--88},
  year={1997},
  publisher={Elsevier}
}

@article{W99,
  title={Newton's law of cooling},
  author={Winterton, R. H. S. },
  journal= CP,
  volume={40},
  number={3},
  pages={205--212},
  year={1999},
  publisher={Taylor \& Francis}
}

@techreport{J10,
  title={Mod{\'e}lisation, analyse math{\'e}matique et simulation num{\'e}rique de la dynamique des glaciers},
  author={Jouvet, G.},
  year={2010},
  institution={EPFL}
}

@article{PRR04,
  title={Numerical simulation of the motion of a two-dimensional glacier},
  author={Picasso, M. and Rappaz, J. and Reist, A. and Funk, M. and Blatter, H.},
  journal=IJNME,
  volume={60},
  number={5},
  pages={995--1009},
  year={2004},
  publisher={Wiley Online Library}
}

@article{PRR08,
  title={Numerical simulation of the motion of a three-dimensional glacier},
  author={Picasso, M. and Rappaz, J. and Reist, A.},
  journal=AMBP,
  volume={15},
  pages={1--28},
  year={2008}
}

@book{L77,
  title={Probability {T}heory {I}, 4-th Edition},
  author={Lo\`eve, M.},
  year={1977},
  publisher={Springer, New York}
}

@book{L78,
  title={Probability {T}heory {II}},
  author={Lo\`eve, M.},
  series={Graduate {T}exts in {M}athematics},
  volume={46},
  pages={15},
  year={1978},
  publisher={Springer}
}

@article{GSR18,
  title={A posteriori error analysis of the {C}rank--{N}icolson finite element method for linear parabolic interface problems: A reconstruction approach},
  author={Gupta, J. S. and Sinha, R. K. and Reddy, G. M. M. and Jain, J.},
  journal=JCAM,
  volume={340},
  pages={173--190},
  year={2018},
  publisher={Elsevier}
}

@article{BC02,
  title={On solving elliptic stochastic partial differential equations},
  author={Babu{\v{s}}ka, I. and Chatzipantelidis, P.},
  journal=CMAME,
  volume={191},
  number={37-38},
  pages={4093--4122},
  year={2002},
  publisher={Elsevier}
}

@book{GS03,
  title={Stochastic {F}inite {E}lements: {A} {S}pectral {A}pproach},
  author={Ghanem, R. G. and Spanos, P. D.},
  year={2003},
  publisher={Courier Corporation}
}

@article{B11,
  title={The cooling law and the search for a good temperature scale, from Newton to Dalton},
  author={Besson, U.},
  journal=EJP,
  volume={32},
  number={2},
  pages={343},
  year={2011},
  publisher={IOP Publishing}
}

@article{W01,
  title={Early study of heat transfer: {N}ewton and {F}ourier},
  author={Winterton, R. H. S.},
  journal={Heat Transf. Eng},
  volume={22},
  pages={3--11},
  year={2001}
}

@article{EGSZ15,
  title={A convergent adaptive stochastic {G}alerkin finite element method with quasi-optimal spatial meshes},
  author={Eigel, M. and Gittelson, C. J. and Schwab, C. and Zander, E.},
  journal=EMMNA,
  volume={49},
  number={5},
  pages={1367--1398},
  year={2015},
  publisher={EDP Sciences}
}

@article{DKLS16,
  title={Multilevel higher order {QMC} {P}etrov--{G}alerkin discretization for affine parametric operator equations},
  author={Dick, J. and Kuo, F. Y. and Le Gia, Q. T. and Schwab, C.},
  journal=SINUM,
  volume={54},
  number={4},
  pages={2541--2568},
  year={2016},
  publisher={SIAM}
}

@book{W08,
  title={Karhunen-{L}o\`eve {E}xpansions and {T}heir {A}pplications},
  author={Wang, Limin},
  year={2008},
  publisher={London School of Economics and Political Science (United Kingdom)}
}

@article{SGR89,
  title={Stochastic finite element expansion for random media},
  author={Spanos, Pol D and Ghanem, Roger},
  journal={Journal of {E}ngineering {M}echanics},
  volume={115},
  number={5},
  pages={1035--1053},
  year={1989},
  publisher={American Society of Civil Engineers}
}
\bibliographystyle{abbrv}
\section*{Appendix}
Here, we recall some standard interpolation results from \cite{LM06,SZ90}.
Let $\Pi^n \colon H^1(D) \rightarrow V^n$ denote the Cl\'ement-type interpolation operator. For a finite element space of polynomial degree $l$, the following interpolation estimates hold for all $j \leq l+1 \colon$
\begin{align}
   &\|h_n^{-j}( v - \Pi^n v)\| \leq C_{\Pi,1} |v|_j,\;\; \label{I1}\\
   &\| h_n^{\frac{1}{2}-j}(v-\Pi^nv)\|_{\Sigma_n} \leq C_{\Pi,2}  |v|_j .\;\;  \label{I2}
\end{align}
These inequalities are valid for sufficiently smooth functions $v$, where the constants $C_{\Pi,1}$ and $C_{\Pi,2}$ depend solely on the shape-regularity of the underlying family of triangulations.

In addition, associated with the finest common coarsening of $\mathcal{T}_n$ and $\mathcal{T}_{n-1}$, $\hat{\mathcal{T}}:=\mathcal{T}_n \wedge \mathcal{T}_{n-1}$, we introduce the Scott--Zhang interpolator $\hat{\Pi}^n$, with mesh size given by $\hat{h}_n:=\max(h_n, h_{n-1})$. For this operator, the following estimate holds:
\begin{align}
    \|\hat{h}_n^{\frac{1}{2}-j}(v-\hat{\Pi}^nv)\|_{{\Sigma_n \cup \Sigma_{n-1}} \setminus {\Sigma_n \cap \Sigma_{n-1}}} \leq C_{\Pi,3}  |v|_j,\;\;  
\end{align}
where the constant $C_{\Pi,3}$ depends on the shape-regularity of the triangulations as well as the number of mesh refinements.

\end{document}